\documentclass[10pt,oneside,reqno]{amsart}
\usepackage{hyperref, setspace, enumitem}
\usepackage[margin=1.3in]{geometry}
\usepackage{amsfonts,amsmath,amssymb,amsthm, mathrsfs, todonotes}
\usepackage{mathtools}
\usepackage{parskip, enumitem} 

\newtheorem{theorem}{Theorem}[section]
\newtheorem{lemma}[theorem]{Lemma}
\newtheorem{proposition}[theorem]{Proposition}

\theoremstyle{definition}
\newtheorem{definition}[theorem]{Definition}
\newtheorem{example}[theorem]{Example}

\newtheorem{corollary}[theorem]{Corollary}
\newtheorem*{definition*}{Definition}

\theoremstyle{remark}
\newtheorem{remark}[theorem]{Remark}

\newcommand{\calA}{{\mathcal{A}}}
\newcommand{\calB}{{\mathcal{B}}}

\newcommand{\calH}{{\mathcal{H}}}

\newcommand{\scrD}{{\mathscr D}}

\newcommand{\ip}[1]{\left\langle #1 \right\rangle}
\newcommand{\norm}[1]{\left\lVert #1 \right\rVert}
\newcommand{\abs}[1]{\left\lvert #1 \right\rvert}

\renewcommand{\le}{\leqslant}
\renewcommand{\ge}{\geqslant}

\newcommand{\R}{\mathbb R}

\newcommand{\Z}{\mathbb Z}

\newcommand{\C}{\mathbb C}
\newcommand{\D}{\mathbb D}
\newcommand{\T}{\mathbb T}

\numberwithin{equation}{section}

\begin{document}

\title[Weighted Dirichlet type spaces with poly-superharmonic weights]{Higher order weighted Dirichlet type spaces with poly-superharmonic weights and Dirichlet type operators of finite order}

\author{Ashish Kujur}
\address{School of Mathematics \\
Indian Institute of Science Education and Research \\
Thiruvananthapuram.}
\email{ashishkujur23@iisertvm.ac.in}
\thanks{}

\author{Md. Ramiz Reza}
\address{School of Mathematics \\
Indian Institute of Science Education and Research \\
Thiruvananthapuram.}
\email{ramiz@iisertvm.ac.in}
\thanks{}


\begin{abstract}
We study higher-order weighted Dirichlet-type spaces on the unit disc associated with a class of poly-superharmonic weights. A higher-order Littlewood–Paley formula is established enabling the computation of higher-order weighted Dirichlet integrals and allowing us to relate iterates of the Laplacian of the weight to higher-order defect operators of the shift operator on these spaces. This leads to the introduction of Dirichlet-type operators of finite order, a class containing $m$-isometries as well as completely hyperexpansive and completely hypercontractive operators of finite order. We prove that every cyclic operator in this class admits a functional model as the shift on a suitable higher-order weighted Dirichlet-type space, thereby providing a unified extension of the model theories for cyclic completely hyperexpansive operators and cyclic $m$-isometries.
\end{abstract}

\maketitle

\section{Introduction and main results}

Let $\D$ denote the open unit disc, $\T$ the unit circle in the complex plane $\C$, and for $x\in \R$, let $\mathbb Z_{\geqslant x} := \mathbb Z \cap [x,\infty).$ We write $\operatorname{Hol}(\D)$ for the space of holomorphic functions on $\D$. For $k \in \mathbb Z_{\geqslant 1}$, a non-negative integrable function $w$ on $\D$, and $f\in \operatorname{Hol}(\D)$, the \emph{weighted Dirichlet integral of order $k$} is defined by  
\[
D_{w,k}(f) := \int_{\D} |f^{(k)}(z)|^2 w(z)\, dA(z),
\]
where $f^{(k)}$ denotes the $k$-th derivative of $f$ and $dA=\frac{1}{\pi} dx dy$ the normalized Lebesgue area measure on $\D$. The associated weighted Dirichlet-type space $\mathscr D_{w,k}$ consists of those $f\in \operatorname{Hol}(\D)$ with $D_{w,k}(f)<\infty$.

Weighted Dirichlet-type spaces $\mathscr D_{w,k}$ provide a flexible function-theoretic framework that has been extensively studied in connection with classical analytic function spaces and operator theory. For $k=1,$ such spaces arise naturally in the work of Richter \cite{Richter91}, where weights given by positive harmonic functions are used to model cyclic analytic $2$-isometries; see also \cite{RS1991LDI, RSMULTINV, SARA1997, Primerbook}.  This framework was subsequently extended by Aleman \cite{1993multiplication}, who considered weights given by positive superharmonic functions and showed that the corresponding spaces yield models for cyclic, analytic, completely hyperexpansive operators; see also \cite{zbMATH01797590, zbMATH06544404, MR2019HM, zbMATH07984371, zbMATH07962947} and references therein for works on super-harmonically weighted Dirichlet spaces. Recently, Rydhe \cite{Rydhe19} demonstrated that weighted Dirichlet spaces $\mathscr D_{w,k}$ weighted by polyharmonic functions play a central role in the model theory of cyclic $m$-isometries, showing that polyharmonic weights encode the structure of higher-order isometric behavior in a natural and intrinsic way, see also \cite{LuoRichter,GGR22,GGRLDF,LuoRydhe25}. Motivated by this progression, we initiate the study of poly-superharmonically weighted Dirichlet-type spaces for arbitrary orders of $k\geqslant 1.$ More specifically, we consider positive $k$-superharmonic weights $w,$ that is,
\begin{align*}
   (-1)^k\triangle_z^k(w(z)) \geqslant 0 \qquad \mbox{(in the distributional sense)},    
\end{align*}
where $\triangle_z$ denotes the complex Laplacian, and investigate the associated spaces $\mathscr D_{w,k}.$ 

In what follows, we provide examples of classical function spaces which arise as a weighted Dirichlet-type space $\mathscr D_{w,k}$ induced by poly-superharmonic weights. For each $\alpha \in \mathbb R,$ let $\mathscr D_{\alpha}$ denote the subspace of $\operatorname{Hol}(\D)$ given by
\begin{equation}
\mathscr D_{\alpha}:= \bigg\{f=\sum_{k=0}^\infty a_k z^k \,\in \operatorname{Hol}(\D)  \, \Big| \,  \|f\|^2_{_{\mathscr D_{\alpha}}} := \sum_{k=0}^\infty(k+1)^{\alpha} |a_k|^2 < \infty \bigg\},
\label{eq:D-alpha}
\end{equation}
see \cite{Taylor66} for a study of $\mathscr D_{\alpha}$ spaces. For $\alpha <0,$ an equivalent norm on  $\mathscr D_{\alpha}$ is given by 
\begin{align*}
    \|f\|^2:= \int_{\mathbb D} |f(z)|^2 (1-|z|^2)^{-\alpha-1}dA(z), 
\end{align*}
see \cite[Lemma 2]{Taylor66}. Now, consider $\alpha >0$ and assume that $\alpha\in (k-1,k]$ for some $k\in \mathbb Z_{\geqslant 1}.$ Since $f\in \mathscr D_{\alpha}$ if and only if $f'\in \mathscr D_{\alpha -2},$ it follows that $f\in \mathscr D_{\alpha}$ if and only if 
\begin{align*}
    \int_\mathbb D |f^{(k)}(z)|^2 (1-|z|^2)^{2k-\alpha-1} dA(z) < \infty.
\end{align*} This gives us that $\mathscr D_{\alpha}= \mathscr D_{w,k}$ as a set where $w$ is given by $w(z)= (1-|z|^2)^{2k-\alpha-1},$ for $z\in\D.$ A straightforward computation will give us that $w$ is a positive $k$-superharmonic function on $\D$ for $\alpha\in (k-1,k].$ Moreover for $\alpha=k,$ the function $w$ is a $k$-harmonic function on $\D,$ that is, $\triangle_z^{k}(w(z))=0$ on $\D.$ This motivates us to study the higher order weighted Dirichlet spaces $\mathscr D_{w,k},$ where $w$ is a positive $k$-superharmonic function on $\D.$

A fundamental class of $k$-superharmonic function on $\D$ is given by the generalized Green potentials associated with measures supported on $\overline{\mathbb D}=\D \cup \T.$ Let $\mu$ be a finite positive Borel measure supported on $\overline{\D}$, and decompose $\mu=\mu_1+\mu_2$ with $\mu_1=\mu|_{\D}$ and $\mu_2=\mu|_{\T}$. The corresponding generalized Green potential $U_{\mu,k}$ is defined by
\begin{align*}
    U_{\mu,k}(z) &:= \frac{1}{((k-1)!)^{2}} 
    \int_{\D} \frac{(-1)^{k+1} G_k(z,\zeta)}{(1-|\zeta|^{2})^{k}}\, d\mu_1(\zeta) 
    + \frac{1}{k!(k-1)!} \int_{\T} \frac{(1-|z|^{2})^{k}}{|z-\zeta|^{2}}\, d\mu_2(\zeta),\quad z\in \D,
\end{align*}
where $G_k(z,\zeta)$ is the polyharmonic Green–Almansi function
for $z, \zeta \in \overline{\mathbb{D}}$, $z \ne \zeta$  given by
\begin{align*}
G_k(z, \zeta) := |\zeta - z|^{2(k-1)} \log \left| \frac{1 - z\overline{\zeta}}{\zeta - z} \right|^2     + \sum_{l=1}^{k-1} \frac{(-1)^{l}}{l} |\zeta - z|^{2(k-1-l)} (1 - |\zeta|^2)^{l} (1 - |z|^2)^{l},
 \end{align*}
see \cite[Page 86]{Begehr07}. Then $U_{\mu,k}$ is a positive $k$-superharmonic function on $\D,$ and it is $k$-harmonic function on $\D$ if and only $\mu$ is supported on $\T,$ that is $\mu_1=0,$ see Section 4 for the details. 

In this article, we study properties of the weighted Dirichlet integral $D_{w,k}(f)$ and the associated weighted Dirichlet-type space $\mathscr D_{w,k}$ when $w=U_{\mu,k}.$ For notational convenience, we write $D_{\mu,k}(f)$ in place of $D_{w,k}(f)$ and $\mathscr D_{\mu,k}$ for the corresponding function space $\mathscr D_{w,k}$ when $w=U_{\mu,k}$.

 An important special case arises when $\mu=\delta_{\zeta},$ the Dirac measure at $\zeta\in \overline{\D}.$ In this case the corresponding weighted Dirichlet integral $D_{\delta_{\zeta},k}(f)$ is referred to as the \textit{local Dirichlet integral of order} $k$ for $f$ at $\zeta$ and the associated weighted Dirichlet space $\mathscr D_{\delta_{\zeta},k}$ is called the\textit{ local Dirichlet space of order} $k$ at $\zeta.$ For simplicity, we will often write $D_{\zeta,k}(f)$ and $\mathscr D_{\zeta,k}$ in place of $D_{\delta_{\zeta},k}(f)$ and $\mathscr D_{\delta_{\zeta},k}$ respectively. Finally, for a subset $X\subset \C$, we denote by $M_+(X)$ the class of all finite non-negative Borel measures supported on $X$.

For any finite positive Borel measure $\mu$ supported on $\overline{\D},$ by Fubini's Theorem, it follows that
\begin{align*}
     D_{\mu,k}(f)= \int_{\overline{\D}} D_{\zeta,k}(f) d\mu(\zeta).
\end{align*}
Thus, the study of weighted Dirichlet integrals $D_{\mu,k}(f)$ naturally reduces to the analysis of the local Dirichlet integrals $D_{\zeta,k}(f).$ For $\zeta\in \D,$ the local Dirichlet space $\mathscr D_{\zeta,1}$ coincides with the classical Hardy space $H^2,$ and by an application of Littlewood--Paley formula for the Hardy space, it follows that for any $f\in \operatorname{Hol} \left( \D \right),$ and $\zeta\in \D,$
\begin{align*}
  D_{\zeta,1}(f)= \left\| \frac{f(z)-f(\zeta)}{z-\zeta}\right \|_{H^2}^2,  
\end{align*}
see \cite[IV, Theorem 1.9]{1993multiplication}, \cite[VI, Lemma 3.1]{GarnettBAF}. Our first main result extends this identity to higher orders, yielding a generalized Littlewood–Paley formula for the higher order local Dirichlet integral. To state it, let $\sigma$ denotes the normalized Lebesgue arc-length measure on $\T$ and  set $D_{\sigma,0}(f):= \|f\|^2_{H^2}.$ 
\begin{theorem} \label{thm:Littlewood-Paley-Higher-Order}
    Let $k \in \mathbb Z_{\geqslant 1}$, $\zeta \in \mathbb D$ and $f\in \operatorname{Hol} \left( \D \right).$  Then we have that 
    \begin{equation*}
	       D_{\zeta, k} (f) =  D_{\sigma, k-1} \left( \frac{f(z)-f\left( \zeta \right)}{z-\zeta} \right).
    \end{equation*}
\end{theorem}

We remark that a corresponding identity also holds when $\zeta \in \mathbb{T}.$  In the case $k=1$, this is the classical \emph{local Douglas formula}, established by Richter and Sundberg \cite[Proposition 2.2]{RS1991LDI} (see also \cite[Proposition 1]{SARA1997}). Higher-order analogues for $\zeta \in \mathbb{T}$ were obtained by Ghara et al. in \cite[Theorem 1.1]{GGRLDF}. 

Using the higher-order Littlewood–Paley formula and the local Douglas formula, we next analyze the multiplication operator $M_z$ acting on the weighted Dirichlet-type space $\mathscr{D}_{\mu,k}.$
To describe its structure, recall that for a non-negative integer $n\in \mathbb Z_{\geqslant 0}$ and a bounded operator $T$ acting on a Hilbert space $\mathcal H,$  the $n$-th backward difference $B_n(T)$ is defined by
\begin{align*}
    B_n(T)= \sum\limits_{j=0}^n (-1)^j\binom{n}{j}{T^*}^j T^j.
\end{align*} 
Let $\mu\in M_+(\overline{\D})$ and $k\in \mathbb Z_{\geqslant 1}.$ We find that the operator $M_z$ on $\mathscr D_{\mu,k}$ satisfies the following operator inequalities: 
\begin{align*}
    (-1)^kB_{n} \left( M_z \right) \ge 0, \,\,\mbox{for every\,\,} n \ge k,
\end{align*}
 Moreover, the operator $M_z$ on $\mathscr D_{\mu,k}$ is a $(k+1)$-isometry, that is, $B_{k+1} \left( M_z \right)=0$ if and only if the measure $\mu$ is supported on $\T,$ see  Proposition \ref{thm:shift-completely-hyperstuff}. In particular, $M_z$ is completely hyperexpansive operator of order $k$ when $k$ is odd and completely hypercontractive operator of order $k$ when $k$ is even, a notion studied by Chavan and Sholapurkar in \cite[Sec.~3]{CS17} and \cite[Sec.~4]{CMF}. Motivated by this dichotomy, and following the terminology of \cite[IV, Definition~1.2]{1993multiplication}, we introduce the notion of Dirichlet-type operators of finite order.

\begin{definition*}[Dirichlet type operator of finite order]
Let $k\in\mathbb Z_{\geqslant 1}.$ An operator $T$ in $\calB \left( \calH \right)$ is said to be a Dirichlet type operator of order $k$ if $(-1)^kB_{n} \left( T \right) \ge 0$ for all $n \ge k.$ 
\end{definition*}

Therefore, the operator $M_z$ on $\mathscr D_{\mu,k}$ is a Dirichlet type operator of order $k$ for any $\mu\in M_+(\overline{\D})$ and $k\in\mathbb Z_{\geqslant 1}.$ For another class of examples, note that for $k\in\mathbb Z_{\geqslant 1},$  the operator $M_z$ on $\mathscr D_{\alpha}$ is of Dirichlet type of order $k$ if $\alpha\in (k-1,k],$ see Remark \ref{another example}. It also follows from \cite[Theorem 2.5]{2015Gu}, that every $(k+1)$-isometry is necessarily a Dirichlet type operator of order $k.$ 

Motivated by the influential work of Agler and Stankus on $m$-isometry in \cite{Agler1,Agler2,Agler3}, Rydhe provided a model for every cyclic $m$-isometry on a Hilbert space in terms of shift operator on higher order weighted Dirichlet type spaces, see \cite[Theorem 3.1]{Rydhe19}. In this article, we develop an analogous model for cyclic Dirichlet-type operators of finite order. 

To state the result, let $\mathcal D$ be the Fr\'{e}chet space of $C^{\infty}$ functions on $\T$ and $\mathcal D'$ be the topological dual of $\mathcal D,$  the space of distributions on the unit circle $\mathbb T.$ For each \textit{allowable normalized} tuple $\overrightarrow{\mu}= (\mu_0,\mu_1,\ldots,\mu_{m})\in (\mathcal D')^{m}\times M_+(\overline{\D}),$ we introduce a  Hilbert space $\mathscr D_{\overrightarrow{\mu}}$ which serves as a model space for realizing a cyclic Dirichlet type operators of finite order. We refer the reader to Section 5 for the terminologies and the details. The second main result of this article is the following model theorem for cyclic Dirichlet type operators of finite order.

\begin{theorem}\label{Model Theorem}
    Let $m\in\mathbb Z_{\geqslant 1}.$ If $\overrightarrow{\mu}= (\mu_0,\mu_1,\ldots,\mu_{m})\in (\mathcal D')^{m}\times M_+(\overline{\D})$ is a normalized allowable $(m+1)$-tuple, then the operator $M_z$ on  $\mathscr D_{\overrightarrow{\mu}}$ is a cyclic Dirichlet type operator of order $m.$ Conversely if $T\in \mathcal B(\mathcal H)$ is a Dirichlet type operator of order $m$ with cyclic unit vector $e,$ then the tuple $\overrightarrow{\mu}= (\mu_0,\mu_1,\ldots,\mu_{m})\in (\mathcal D')^{m}\times M_+(\overline{\D})$ given by 
    \begin{align*}
        \hat{\mu_j}(k)= \overline{\hat{\mu_j}(-k)}&= \langle \beta_j(T)e, T^ke\rangle_{\mathcal H},\quad k\in \mathbb Z_{\geqslant 0} , \quad j=0,1,\ldots,m-1,\\
         \int_{\mathbb{\overline{\D}}} z^{k} \overline{z}^{l} \mathrm{d}\mu_{m} (z) &= \ip{\beta_{m} (T) T^k e, T^l e}_{\calH},
    \end{align*} is a normalized allowable $(m+1)$-tuple, and there exists a unitary $U:\mathcal H \to \mathscr D_{\overrightarrow{\mu}}$ such that $UT=M_zU$ and $Ue=1.$
    
    Moreover if $T_j\in \mathcal B(\mathcal H_j),$ for $j=1,2,$ are cyclic and  Dirichlet type operators of order $m$ with cyclic vector $e_j,$ then the associated normalized allowable $(m+1)$-tuple coincides if and only if there exists a unitary $V:\mathcal H_1\to \mathcal H_2$ such that $VT_1=T_2V$ and $Ve_1=e_2.$ 
\end{theorem}

We remark that if $\overrightarrow{\mu}= (\mu_0,\mu_1,\ldots,\mu_{m})\in (\mathcal D')^{m}\times M_+(\overline{\D})$ is a normalized allowable $(m+1)$-tuple, then the operator $M_z$ on  $\mathscr D_{\overrightarrow{\mu}}$ is an $(m+1)$-isometry if and only if $\mu_m$ is supported on $\T,$ that is, $\mu_m\in M_+(\T).$ And in that case the model spaces $\mathscr D_{\overrightarrow{\mu}}$ coincides with the model spaces for cyclic $(m+1)$-isometries as described by Rydhe in \cite[Theorem 3.1]{Rydhe19}. Moreover the case of $m=1,$ gives us the model for every cyclic completely hyperexpansive operators on a Hilbert spaces extending the model theorem by S. Richter in \cite[Theorem 5.1]{Richter91} and A. Aleman in \cite[IV, Theorem 2.5]{1993multiplication}. In this way, our results provide a natural extension and unification of the existing model theories for cyclic completely hyperexpansive operators, and cyclic $m$-isometries within a single higher-order framework.

The organization of the paper is as follows. In Section~2, we collect basic properties of the polyharmonic Green–Almansi function $G_k(z,\zeta)$ and use them to derive key properties of the associated weighted Dirichlet integrals $D_{\zeta,k}(f)$ for $\zeta \in \mathbb{D}.$

In Section~3, we introduce the local Dirichlet spaces $\mathscr D_{\zeta,k}$ for $\zeta \in \mathbb{D}$, and establish a higher-order version of the Littlewood–Paley formula for $D_{\zeta,k}(f).$

Section~4 is devoted to the study of higher-order weighted Dirichlet-type spaces $\mathscr{D}{\mu,k}$ associated with measures $\mu \in M^+(\overline{\mathbb{D}}),$ including an analysis of the multiplication operator $M_z.$

Finally, in Section~5, we introduce the model spaces $\mathscr{D}_{\overrightarrow{\mu}}$ and prove the model theorem for cyclic Dirichlet-type operators of finite order.

\section{Preliminaries on polyharmonic Green function}
We will be using the following notation throughout the article. For two family of parametrized set $\left\{A_{i,j}: i\in I,j\in J\right\}$, $
\left \{B_{i,j}: i\in I,j\in J\right \} \subseteq [0,\infty),$ we write $A_{i,j} \lesssim_{j} B_{i,j}$ to denote that there exists constants $C_j >0$ (independent of $i$) such that $A_{i,j} \leqslant C_j B_{i,j}$ for every $i\in I,j\in J.$ If $A_{i,j} \lesssim_{j} B_{i,j}$ and $B_{i,j} \lesssim_{j} A_{i,j},$ we write $A_{i,j} \thickapprox_{j} B_{i,j}.$ We define the Laplacian to be 
\begin{equation*}
    \Delta_{z} = \frac{\partial^2}{\partial z \partial \overline{z}} = \frac{1}{4} \left( \frac{\partial ^{2}}{\partial x^{2}} + \frac{\partial ^{2}}{\partial y^{2}} \right), \qquad  z=x+iy.
\end{equation*}

Following \cite[Definition 1]{Begehr2002}, we consider the polyharmonic Green-Almansi function $G_k(z,\zeta)$ for $z, \zeta \in \overline{\mathbb{D}}$, $z \ne \zeta$ and $k \ge 1$, given by
\begin{equation*}
G_k(z, \zeta) = |\zeta - z|^{2(k-1)} \log \left| \frac{1 - z\overline{\zeta}}{\zeta - z} \right|^2 + \sum_{l=1}^{k-1} \frac{(-1)^{l}}{l} |\zeta - z|^{2(k-1-l)} (1 - |\zeta|^2)^{l} (1 - |z|^2)^{l}.
 \end{equation*}
We note down a few important properties of the polyharmonic Green-Almansi function $G_k$ on the unit disc $\mathbb D,$ see \cite[Page 86]{Begehr07}). 

\begin{proposition}\label{Basic properties of Green's function}
    The following statements holds.
\begin{enumerate}
    \item[(a)] For a fixed $\zeta \in \D$, $G_k (z, \zeta)$ is $k$-harmonic, that is, $\Delta_z^k(G_k(z,\zeta))=0,$ for $z \in \D \setminus \{ \zeta \},$
    \item[(b)] For a fixed  $\zeta \in \D$, $G_k (z, \zeta) + \abs{\zeta - z}^{2(k-1)} \log \abs{\zeta - z}^2$ is $k$-harmonic in $z \in \D$, 
    \item[(c)] For a fixed  $\zeta \in \D$, $\Delta_z ^j G_k (z, \zeta) = 0$ for $0 \le 2j \le k-1$ for $z \in \partial\D$ and $\zeta \in \D$,
    \item[(d)] $\frac{\partial}{\partial n} \Delta_z ^j G_k (z, \zeta) = 0$ for $0\le 2j \le k-2$ for $z\in \T$ and $\zeta \in \D$,
    \item[(e)] $G_k (z, \zeta) = G_k (\zeta , z)$ for $z, \zeta \in \D$ with $z\ne \zeta$.
    \item[(f)] $\triangle_z^k G_k(z,\zeta)= -((k-1)!)^2 \delta_{\zeta}$ (in distributional sense) for each $\zeta \in \mathbb D$.
\end{enumerate}
\end{proposition}

We remark that item $(f)$ of the previous proposition follows from the fact that 
\begin{equation*}
    \Delta_{z}^{k} \left( \int_{\D}  G_k (z, \zeta) g(\zeta) \, dA(\zeta) \right) = - ((k-1)!)^2g(z)
\end{equation*}
for each $g \in C_c^{\infty} (\D),$ the set of compactly supported smooth functions on $\D,$ see \cite[Page~95]{Begehr07}.

In the following proposition, we establish identities and estimates for the Green function $G_k$, which will be useful later in the article. While the bi-harmonic Green's function has been extensively studied (see \cite{Garbedian64, Begehr2002, DurenBergman, TheoryofBergmanJhu, AH2002} and the references therein), not many results for its higher-order counterparts appear in the literature, to the best of our knowledge. We therefore present and prove some of the analogous results for the higher-order case for the sake of completion. 

\begin{proposition}
The following statements are true:
\begin{enumerate}
    \item For $k\ge 2$, $z \in \D$ and $\zeta \in \overline{\mathbb{D}}$, we have 
	\begin{equation*}
	    G_{k} \left( z, \zeta \right) = \abs{z-\zeta}^{2} G_{k-1} \left( z, \zeta \right) + \frac{\left( -1 \right)^{k-1}}{k-1} \left( 1-\abs{z}^{2} \right)^{k-1} \left( 1-\abs{\zeta}^{2} \right)^{k-1}.
	\end{equation*}

    \item Let $\varphi_{a} \left( w \right) = \frac{a-w}{1-\overline{a}w}, a\in \D, w \in \D$. Then we have that 
	\begin{equation}
	    G_{k} \left( \varphi_{a} \left( z \right) , \varphi_{a} \left( \zeta \right) \right) = \left( \frac{1-\abs{a}^{2}}{\abs{1-\overline{a}z} \abs{1-\overline{a}\zeta}} \right)^{2\left( k-1 \right)} G_{k} \left( z, \zeta \right), \qquad  z, \zeta \in \D.
	    \label{eqn:Green-Mobius}
	\end{equation}
    \item  For $k\ge 2$, $z \in \D$ and $\zeta \in \mathbb{D}$, 
	\begin{equation}
	    \frac{1}{k} \frac{ \left( \left( 1-\abs{\zeta}^{2} \right) \left( 1- \abs{z}^{2} \right)\right) ^{k} }{\abs{1-\overline{\zeta}z}^{2}} \le \left( -1 \right)^{k+1} G_{k} \left( z, \zeta \right) \le \frac{1}{k-1}  \frac{ \left( \left( 1-\abs{\zeta}^{2} \right) \left( 1- \abs{z}^{2} \right)\right) ^{k} }{\abs{1-\overline{\zeta}z}^{2}}.
        \label{eqn:Green-estimate}
	\end{equation}
	Hence, we have that for each $\zeta,z \in \D$ and $k \ge 2$, the $\left( -1 \right)^{k+1} G_{k} \left( z, \zeta \right)  \ge 0$ and 
	\begin{equation}
	    \left( -1 \right)^{k+1} G_{k} \left( z, \zeta \right) \thickapprox_{\zeta, k} \left( 1- \abs{z}^{2} \right)^{k}, \qquad  z\in \mathbb D.
	    \label{eqn:G_k-growth}
	\end{equation}
    \item For $k\ge 2$, $z \in \D$ and $\zeta \in \overline{\mathbb{D}}$, we have
    \begin{equation*}	
	\Delta_{z} G_{k} \left( z, \zeta \right) = \left( k-1 \right)^{2} G _{k-1} \left( z, \zeta \right) - \left( k-1 \right) H_{k-1} \left( z, \zeta \right),
    \end{equation*} where the function $H_n(z,\zeta)$ for $z\in\D$ and $\zeta\in\overline{\D}$ is given by
    \begin{equation*}
    H_{n} \left( z, \zeta \right) := \left (-1\right)^{n-1}\left( 1-\abs{\zeta}^{2} \right)^{n} \left( 1-\abs{z}^{2}  \right)^{n-1}  \left( \frac{1-\abs{z\zeta}^{2}}{\abs{1-\overline{z}\zeta}^{2}} \right),\quad  n\in \mathbb Z_{\geqslant 1}.
\end{equation*}
    \end{enumerate}
    
    \label{prop:properties-of-Green-functions}
\end{proposition}

\begin{proof}
    \begin{enumerate}
	\item For $k \ge 2$, observe that 
	    \begin{align*}
		& G_{k} \left( z, \zeta \right) - \frac{\left( -1 \right)^{k-1}}{k-1} \left( 1- \abs{z}^{2} \right)^{k-1} \left( 1- \abs{\zeta}^{2} \right)^{k-1} \\
&= |\zeta - z|^{2(k-1)} \log \left| \frac{1 - z\overline{\zeta}}{\zeta - z} \right|^2 + \sum_{l=1}^{k-1} \frac{(-1)^{l}}{l} |\zeta - z|^{2(k-1-l)}  (1 - |\zeta|^2)^{l} (1 - |z|^2)^{l}   \\ 
& \hspace{50mm} - \frac{\left( -1 \right)^{k-1}}{k-1} \left( 1- \abs{z}^{2} \right)^{k-1} \left( 1- \abs{\zeta}^{2} \right)^{k-1} \\
 &= |\zeta - z|^{2(k-1)} \log \left| \frac{1 - z\overline{\zeta}}{\zeta - z} \right|^2 + \sum_{l=1}^{k-2} \frac{(-1)^{l}}{l} |\zeta - z|^{2(k-1-l)}  (1 - |\zeta|^2)^{l} (1 - |z|^2)^{l}  \\
 &= \abs{\zeta - z}^{2} \left( |\zeta - z|^{2(k-2)} \log \left| \frac{1 - z\overline{\zeta}}{\zeta - z} \right|^2 + \sum_{l=1}^{k-2} \frac{(-1)^{l}}{l} |\zeta - z|^{2(k-2-l)}  (1 - |\zeta|^2)^{l} (1 - |z|^2)^{l}  \right) \\
 &= \abs{z-\zeta}^{2} G_{k-1} \left( z, \zeta \right).
	    \end{align*}
	\item This follows by direct substitution using the identities below:
	    \begin{align*}
		\varphi_{a} \left( z \right) - \varphi_{a} \left( \zeta \right) &= \frac{( 1-\abs{a}^{2} ) \left( z-\zeta \right)}{\left( 1-\overline{a}z \right) \left( 1-\overline{a}\zeta \right)},\qquad\abs{\frac{1-\varphi_{a} \left( z \right) \overline{\varphi_{a} \left( \zeta \right)}}{\varphi_{a} \left( z \right) - \varphi_{a} \left( \zeta \right)}} = \abs{\frac{1-z\overline{\zeta}}{z-\zeta}}, \\
		  1-\abs{\varphi_{a} \left( z \right)}^{2} &= \frac{\left( 1-\abs{a}^{2} \right) \left( 1-\abs{z}^{2} \right)}{\abs{1-\overline{a}z}^{2}},\,\,\,\qquad a,z,\zeta \in \D.
	    \end{align*}
	    
	\item We prove it for the case $\zeta =0$. We would be then done by item (2) of this Proposition. For $k \ge 2,$ consider the function $f_{k} : \left( 0, \infty \right) \to \R ,$ given by
	\begin{align*}
	    f_{k}\left(x\right) = 
        -x^{k-1}\ln x+\sum_{n=1}^{k-1}\frac{\left(-1\right)^{n}}{n}x^{k-1-n}\left(1-x\right)^{n},\,\, \qquad  x > 0.
	\end{align*}
    We make two observations at this point: 
    \begin{enumerate}[label=(\roman*)]
        \item $G_{k} \left( z, 0 \right) = f_{k} \left( \abs{z}^{2} \right)$ for $z \in \mathbb D$, and, 
        \item $f_k (x) = x f_{k-1} (x) + \frac{(-1)^{k-1}}{k-1} (1-x)^{k-1}$ for $k \ge 3$ and $x \ge 0$.
    \end{enumerate}
Note that the second observation also implies item $(1)$ of this Proposition. We now claim that 
    \begin{equation*}
        (-1)^{k+1} f_k (1-x) = (k-1)! \sum_{n=0}^{\infty} \frac{x^{n+k}}{(n+1)^{\overline{k}}}, \qquad  k \ge 2 ,\,\, |x| \le 1, x \ne 1 .
    \end{equation*} where $(x)^{\overline{m}}= x(x+1)\cdots(x+m-1)$ denotes the rising factorial. The proof follows using the induction argument. We first consider the base case $k=2$. First observe that
    \begin{equation*}
        f_2 (x) =
	-x \ln x - (1-x) \qquad \left( x > 0 \right).
    \end{equation*}
and consequently, we have that for $x\in [-1, 1)$, we have
    \begin{equation*}
        -f_2 (1-x) = (1-x) \ln (1-x) + x 
                    = (1-x) \left( - \sum_{n=1}^{\infty} \frac{x^{n}}{n} \right) + x
                    = \sum_{n=0}^{\infty} \frac{x^{n+2}}{(n+1)(n+2)}.
    \end{equation*}
Now assume the induction hypothesis holds for $l=k-1 \ge 2$. To prove it for $l=k,$ observe that for $\abs{x} \le 1, x \ne 1$, we have
    \begin{align*}
        (-1)^{k+1} f_{k} (1-x) &= (x-1) \left( (-1)^{k} f_{k-1} (1-x) \right) + \frac{1}{k-1} x^{k-1} \\
        &= (k-2)! (x-1) \sum_{n=0}^{\infty} \frac{x^{n+k-1}}{(n+1)^{\overline{k-1}}} + \frac{x^{k-1}}{k-1} \\
        &= (k-2)! \left[\sum_{n=0}^{\infty} \frac{x^{n+k}}{(n+1)^{\overline{k-1}}} - \sum_{n=1}^{\infty} \frac{x^{n+k-1}}{(n+1)^{\overline{k-1}}} \right]\\
        &= (k-2)! \sum_{n=0}^{\infty} x^{n+k} \left[ \frac{1}{(n+1)^{\overline{k-1}}} - \frac{1}{(n+2)^{\overline{k-1}}} \right] \\
        &= (k-1)! \sum_{n=0}^{\infty} \frac{x^{n+k}}{(n+1)^{\overline{k}}}.
    \end{align*}
    This completes the proof of the claim. Now consider the function $g_{k} : \left( 0, 1 \right) \to \R$ given by $g_{k} \left( x \right) = \frac{\left( -1 \right)^{k+1}f_{k} \left( x \right)}{\left( 1- x \right)^{k}}$. Observe that from the previous claim, we have that 
    \begin{equation*}
        g_k (1-x) = (k-1)! \sum_{n=0}^{\infty} \frac{x^n}{(n+1)(n+2) \cdots (n+k)}.
    \end{equation*}
This shows that $x \mapsto g_k (x)$ is decreasing function on $(0,1)$ and  moreover we have $\lim\limits_{x\to 1^{-}}g_k (x) = \frac{1}{k}$ and $\lim\limits_{x\to 0^{+}} g_k (x) = \frac{1}{k-1}$.
 Thus, we have that 
	\begin{equation*}
	    \frac{1}{k} \left( 1-\abs{z}^{2} \right)^{k} \le \left( -1 \right)^{k+1} G_{k} \left( z, 0 \right) \le \frac{1}{k-1} \left( 1-\abs{z}^{2} \right)^{k},\qquad z\in \D\setminus \{0\}.
	\end{equation*}
The general case follows by using equation \eqref{eqn:Green-Mobius}. We omit the details. The equation \eqref{eqn:G_k-growth}
    follows from \eqref{eqn:Green-estimate} and the following estimate:     \begin{equation*}
	    \frac{1}{(1+\abs{\zeta})^2} \le \frac{1}{\abs{1-\overline{\zeta}z}^2} \le \frac{1}{(1-\abs{\zeta})^2}, \qquad  z, \zeta \in \D.
    \end{equation*}.

    \item See \cite[Lemma 3]{Begehr2002} for a proof.
\end{enumerate}
This completes the proof of the proposition.
    \end{proof}
For a fixed $\zeta \in \mathbb{D}$ and $k \ge 1$, we define  
\begin{equation}
    U_{\zeta, k}(z) =  \frac{(-1)^{k+1} G_k(z, \zeta)}{((k-1)!)^{2} (1-\abs{\zeta}^2)^{k}} , \qquad  z \in \mathbb{D}.
    \label{eqn:local-weight-in-open-disc}
\end{equation}

By Proposition \ref{prop:properties-of-Green-functions}, we have $U_{\zeta, k}(z) > 0$ for all $z \in \mathbb{D}$. Furthermore, for any function $f \colon \D \to \C$, we define the dilated function $f_r$ for each $r \in [0,1]$ as  
\begin{equation*}
    f_r(z) := f(rz), \quad z \in \D.
\end{equation*}
The following proposition generalizes \cite[Proposition 2.4]{ARS96}. 
 \begin{proposition}
\label{prop:analytic-Gk-properties}
Let $f\in \operatorname{Hol} \left( \D \right).$  Then the following properties hold for $k \ge 2$:

\begin{enumerate}[label=(\roman*)]
\item[\rm (i)] For every $\zeta \in \D$, we have the radial limit equality:
\begin{equation*}
\lim_{r\to 1^-} \int_{\D} \Delta_z \left( |f_r(z)|^2 \right) U_{\zeta, k} \, dA(z) = \int_{\D} \Delta_z \left( |f(z)|^2 \right) U_{\zeta, k} \, dA(z).
\end{equation*}

\item[\rm(ii)] The Laplacian-Green's function exchange formula holds:
\begin{equation*}
\int_{\D} \Delta_z \left( |f(z)|^2 \right) U_{\zeta, k} \, dA(z) = \int_{\D} \Delta_z U_{\zeta, k} (z) |f(z)|^2 \, dA(z).
\end{equation*}
\end{enumerate}
  \end{proposition}
\begin{proof}

    \begin{enumerate}[label=(\roman*)]
        \item  In view of \eqref{eqn:G_k-growth}, the proof follows mutatis mutandis as in \cite[Proposition 2.4]{ARS96}; see also \cite[Ch. 9, Lemma 6]{DurenBergman} for a similar version.
        \item Since $G_k$ for $k \geq 2$ enjoys the same properties as $G_2$, specifically, $G_k(z, \zeta) = \frac{\partial G_k}{\partial n}(z, \zeta) = 0$ for each $z \in \T$ and $\zeta \in \D$ (see Proposition \ref{Basic properties of Green's function}, items (c) and (d)) and item (4) of Proposition \ref{prop:properties-of-Green-functions}, the proof in \cite[Proposition~2.4]{ARS96} applies. See also \cite[Ch.3, Proposition~3.18]{TheoryofBergmanJhu}. We omit the details.
    \end{enumerate}
\end{proof}

The following proposition generalizes \cite[Proposition 2.6]{ARS96}.
\begin{proposition}
    If $g$ is a non-negative measurable function on $\D$ then for each $\zeta \in \D$ and $r \in [0,1]$, we have that 
    \begin{equation*}
        r^{2k-1} \int_{\D} g_r (z) U_{\zeta, k} (z) \, dA (z) \le \frac{k}{k-1} \int_{\D} g (z) U_{\zeta, k }(z) \, dA(z).
    \end{equation*}
\end{proposition}
\begin{proof}
We define
\begin{equation*}
    \tilde{H}_k (z,\zeta) :=  \frac{ \left( \left( 1-\abs{\zeta}^{2} \right) \left( 1- \abs{z}^{2} \right) \right) ^{k} }{\abs{1-\overline{\zeta}z}^{2}}, \qquad z, \zeta \in \D.
\end{equation*}
First, we remark that the proof of \cite[Proposition 2.6]{ARS96} shows that for fixed $z, \zeta \in \D$, the function
\begin{equation*}
    r \mapsto r \tilde{H}_2 (z/r, \zeta), \qquad \abs{z} < r < 1,
\end{equation*}
is increasing. A quick computation shows that for $k \ge 2$,
\begin{equation*}
    r^{2k-3} \tilde{H}_k (z/r,\zeta) = \left(r \tilde{H}_2 (z/r, \zeta)\right) (1 - \abs{\zeta}^2)^{k-2}(r^2 - \abs{z}^2)^{k-2}.
\end{equation*}
Consequently, for fixed $z, \zeta \in \D$, the function
\begin{equation*}
    r \mapsto r^{2k-3} \tilde{H}_k (z/r, \zeta), \qquad \abs{z} < r < 1
\end{equation*}
is also increasing. Now observe that for each $\zeta \in \D$, 
\begin{align*}
    r^{2k-1} \int_{\D} g(rz) (-1)^{k+1}G_k(z,\zeta) \, dA(z) & \le \frac{1}{k-1} r^{2k-1} \int_{\D} g(rz) \tilde{H}_{k} (z, \zeta) \, dA (z) \\
    &= \frac{1}{k-1} \int_{r\D} g(z) \left( r^{2k-3} \tilde{H}_{k} (z/r , \zeta) \right) \, dA(z) \\
    &\le \frac{1}{k-1} \int_{r\D} g(z) \tilde{H}_{k} (z, \zeta) \, dA(z) \\
     &\le \frac{1}{k-1} \int_{\D} g(z) \tilde{H}_{k} (z, \zeta) dA(z)\\
    &\le \frac{k}{k-1} \int_{\D} g(z) (-1)^{k+1}G_k(z,\zeta) dA(z).
\end{align*}
We remark that the first and last inequality follows from equation \eqref{eqn:Green-estimate}. This completes the proof.
\end{proof}
The following corollary is now immediate from the previous proposition and will be useful later.

\begin{corollary}
    If $f$ is analytic on $\D$, then for each $\zeta \in \D$ and $r \in [0,1]$, we have
    \begin{equation*}
        \int_{\D} \abs{f_r^{(k)}(z)}^2 U_{\zeta, k}(z) \, dA(z) \le \frac{k}{k-1} \int_{\D} \abs{f^{(k)}(z)}^2 U_{\zeta, k}(z) \, dA(z).
    \end{equation*}
    \label{cor:f_r}
\end{corollary}
\section{Local Dirichlet Spaces and a generalized Littlewood-Paley Formula}

For each $\zeta \in \overline{\D }$ and $k \ge 1$, we consider the family of semi-norm $\sqrt{D _{\zeta, k } \left( \cdot \right)}$ on $\operatorname{Hol} \left( \D \right),$  defined by
\begin{equation*}
   D_{\zeta,k} (f) := \begin{dcases} 
        \frac{1}{( 1-\abs{\zeta}^{2} )^{k} ( ( k-1 )! ) ^{2}} \int_{\D} \abs{f^{\left( k \right)}\left( z\right) }^2 U_{\zeta, k} \left( z \right) \mathrm{d}A (z), & \zeta \in \D \\
        \frac{1}{k! (k-1)!} \int_{\D}  \abs{f^{\left( k \right)}\left( z\right) }^2 P_{\delta_{\zeta}}(z) (1-\abs{z}^2)^{k-1} \, \mathrm{d}A(z), & \zeta \in \T.
   \end{dcases}
\end{equation*}
For $k\in \mathbb Z_{\geqslant 1},$ we define the \textit{local Dirichlet space} $\mathscr D_{\zeta,k }$ by
\begin{equation*}
  \mathscr D_{\zeta,k } := \left\{ f \in \operatorname{Hol} \left( \D \right) \, : \, D _{\zeta, k}(f)< \infty \right\}.
\end{equation*}

We now explore the relationship between the Local Dirichlet space $\mathscr D_{\zeta,k }$ and the classical weighted Hardy space $\mathscr{D} _{\alpha}.$
Recall that for $\alpha \in  \mathbb R$ the $\mathscr{D} _{\alpha}$ space is defined by
\begin{equation*}
    \mathscr{D}_{\alpha} :=  \left\{ f = \sum_{n \ge 0} a_{n} z^{n} \in \operatorname{Hol} \left( \D \right)\, : \, \norm{f}_{\alpha}^{2} = \sum_{n \ge 0} \left( n+1 \right)^{\alpha} \abs{a_{n}}^{2} < \infty  \right\}.
    \end{equation*}
Let $k\in \mathbb Z_{\geqslant 1.}$ We consider the semi-norm $D_{\sigma,k}(f)$ on $\operatorname{Hol}(\D)$ given by 
\begin{align*}
    D_{\sigma,k}(f)= \frac{1}{k! (k-1)!}\int_{\D} |f^{(k)}(z)|^2 (1-|z|^2)^{k-1}\, dA(z).
\end{align*}
Let $f\in \operatorname{Hol}(\D)$ and $f(z)= \sum_{n\geqslant 0} a_n z^n.$ A straightforward computation will give us that 
\begin{align*}
    D_{\sigma,k}(f)= \sum\limits_{n=k}^{\infty} \binom{n}{k}|a_n|^2,
\end{align*} 
see \cite[Lemma 3.2]{Rydhe19}. Therefore it follows that $D_{\sigma,k}(f) < \infty$ if and only if $f\in \mathscr D_k.$  For $k \in \mathbb Z_{\geqslant 1},$  we put an equivalent norm on $\mathscr D _{k},$  by
\begin{equation*}
    \norm{f}_{\mathscr D _{k}} ^{2} = \norm{f}_{H^{2}}^{2} + D_{\sigma, k} \left( f \right), \qquad f \in \mathscr D_{k}.
\end{equation*}

In the next proposition, we find that $\mathscr D_{\zeta,k } =\mathscr D_{k-1}$ for each $k \ge 1$ and $\zeta \in \mathbb D$ and consequently, the following norm on $\mathscr D_{\zeta, k}$ given by
\begin{equation*}
    \norm{f}_{\mathscr D{_{\zeta, k}}} ^{2}  := \norm{f}_{H^{2}}^{2} +D_{\zeta,k} \left( f \right), \qquad f\in \operatorname{Hol} \left( \D \right) .
\end{equation*}
defines an equivalent norm on $\mathscr D_{k-1}$.

\begin{proposition} \label{prop:equivalence-of-norm}
    For each $k \in \mathbb Z_{\geqslant 1}$ and $\zeta \in \mathbb D$, we have that $\mathscr D_{ \zeta, k} = \mathscr{D}_{k-1} \subset H^{2}$ and there are constants $c_{k, \zeta}, C_{k, \zeta} > 0$ such that 
    \begin{equation}
	c_{k, \zeta}  \norm{f}_{\mathscr D{_{\zeta, k}}} ^{2}  \le \norm{f}_{\mathscr{D}_{k-1}}^{2} \le C_{k, \zeta}  \norm{f}_{\mathscr D{_{ \zeta, k}}} ^{2},  \qquad f\in \operatorname{Hol} \left( \D \right).
    \label{eqn:equivalent-norms-on-Dk}
    \end{equation}
\end{proposition}
\begin{proof}
The fact that $\mathscr D_{\zeta, 1} = \scrD_{0}$ follows from \cite[Proposition 2.6]{1993multiplication}. Now, we show that $\mathscr D_{ \zeta, k} = \scrD_{k-1}$ for every $k\ge 2.$ First, we note that an application of \eqref{eqn:G_k-growth} shows that for a holomorphic function $f$ on $\D$, the semi-norm $D_{\zeta, k} (f) < \infty$ if and only if the integral $\int_{\D} \abs{f^{(k)} (z)}^2 (1-\abs{z}^2)^{k} \mathrm{d}A(z) < +\infty.$ To see $\mathscr D_{\zeta, k} \subset \scrD_{k-1}$, take any $f \in \mathscr D_{\zeta, k}$. Consider a $g \in \operatorname{Hol} (\D)$ such that $g'=f$. Since $D_{\zeta,k } (f) < \infty$, we have 
\begin{align*}
    \int_{\D} \abs{f^{(k)} (z)}^2 (1-\abs{z}^2)^{k} \mathrm{d}A(z) = \int_{\D} \abs{g^{(k+1)} (z)}^2 (1-\abs{z}^2)^{k} < +\infty
\end{align*}
 Thus, $g^{(k+1)}\in \scrD_{-k-1},$ see \cite[p. 229]{Taylor66}. Consequently, we have that $f =g' \in \scrD_{k-1}$. The last statement is true because of the fact that $f \in \scrD_{\alpha}$ if and only if $f' \in \scrD_{\alpha-2}$. Therefore we find that $\mathscr D_{\zeta,k}\subset \mathscr D_{k-1}.$ The reverse inclusion follows by reversing the argument.

    We proceed to prove inequalities in \eqref{eqn:equivalent-norms-on-Dk}.
        Fix $\zeta \in \mathbb D$. The case $k=1$ can be found in \cite{1993multiplication} but we sketch the proof here for the sake of completeness. Observe that for $f \in \operatorname{Hol} \left( \D \right)$, we have 
    \begin{align*}
	D_{\zeta, 1}(f) &= \frac{1}{(1-|\zeta|^2)} \int_{\D} \abs{f'\left( z \right)}^{2} G_{1} \left( z, \zeta \right) dA \left( z \right) \\
	&= \norm{\frac{f-f\left( \zeta \right)}{z-\zeta}}_{H^{2}}^{2} \lesssim_{\zeta} \norm{f}_{H^{2}}^{2} = D_{\sigma, 0} \left( f \right).
    \end{align*}
    We remark that the second equality is true by Littlewood-Paley theorem for Hardy spaces (see \cite[Chapter VI, Lemma 3.1]{GarnettBAF}). Consequently, we have that
    \begin{equation*}
	\norm{f}_{H^{2}}^{2} \le \norm{f}_{\mathscr D_{\zeta, 1}}^{2} \lesssim_{\zeta} \norm{f}_{H^{2}}^{2}, \qquad  f \in \operatorname{Hol} \left( \D \right).
    \end{equation*} 
Now assume $k \ge 2$. First note that 
     \begin{equation}
         \frac{1-\abs{\zeta}^2}{(1+\abs{\zeta})^2} \le \frac{1-\abs{z\zeta}^2}{\abs{1-\overline{z}\zeta}^{2}} \le \frac{1}{(1-\abs{\zeta})^2}, \qquad z, \zeta\in \D.
         \label{eqn:bound-Hk}
    \end{equation}
    We also note that that it follows from \eqref{eqn:bound-Hk} that  $D_{\sigma, k-1} (f) < + \infty$ if and only if for each $\zeta \in \D$, $$\int_{\D} \abs{f^{(k-1)} \left( z \right)}^2 \left( (-1)^{k} H_{k-1} \left( z, \zeta \right) \right)\mathrm{d}A\left( z \right) < +\infty.$$
    Now we will establish the first inequality of equation \eqref{eqn:equivalent-norms-on-Dk}. If $D_{\sigma, k-1}(f) = +\infty$, then there is nothing to do. So, we may assume $D_{\sigma, k-1} (f) < +\infty.$  Now to establish the first inequality, take any $f \in \operatorname{Hol} \left( \D \right)$ and consider the following:
    \begin{align*}
	D_{\zeta, k} \left( f \right) &= \frac{1}{(1-\abs{\zeta}^2)^{k} ((k-1)!)^2} \int_{\D} \abs{f ^{\left( k \right)} (z)} ^{2} U_{\zeta, k} (z) \mathrm{d}A(z) \\
	&= \frac{\left( k-1 \right)^{2}}{(1-\abs{\zeta}^2)^{k} ((k-1)!)^2} \int_{\D} \abs{f^{(k-1)} \left( z \right)}^2 \left((-1)^{k+1} G_{k-1} \left( z, \zeta \right)  \right) \mathrm{d}A\left( z \right) \\ & \hspace{10mm}+ \frac{\left( k-1 \right)}{(1-\abs{\zeta}^2)^{k} ((k-1)!)^2} \int_{\D} \abs{f^{(k-1)} \left( z \right)}^2 \left( (-1)^{k} H_{k-1} \left( z, \zeta \right) \right)\mathrm{d}A\left( z \right) \\
	&\le  \frac{\left( k-1 \right)}{(1-\abs{\zeta}^2) ((k-1)!)^2} \int_{\D} \abs{f^{\left( k-1 \right)}\left( z \right)}^{2} \left( 1-\abs{z}^{2} \right)^{k-2} \left( \frac{1-\abs{z\zeta}^{2}}{\abs{1-\overline{z}\zeta}^{2}} \right) \mathrm{d}A\left( z \right) \\
	&\lesssim_{k, \zeta} D_{\sigma, k-1} \left( f \right).
    \end{align*}
    
    We remark that the second equality is true by item $(ii)$ of Proposition \ref{prop:analytic-Gk-properties}. The third inequality holds true because $(-1)^{k+1} G_{k-1} (z, \zeta) \le 0$ by Proposition \ref{prop:properties-of-Green-functions} (3). The last inequality follows from equation \eqref{eqn:bound-Hk}. This establishes our first inequality of \eqref{eqn:equivalent-norms-on-Dk}.

        Now, we proceed to show the second inequality.  We show that $\left( \mathscr{D}_{\zeta, k}, \lVert \cdot \rVert _{\mathscr D_{\zeta, k}} \right)$ is Hilbert space. Then we will obtain the second inequality by virtue of the Bounded Inverse Theorem. However, this proof is similar to that of \cite[Theorem 1.6.3]{Primerbook}, so, we omit the details. 
\end{proof}

\begin{corollary}
    Polynomials are dense in $\mathscr D_{\zeta, k}$.
\end{corollary}
\begin{proof}
    Since polynomials are dense in $\mathscr{D}_{k}$ spaces, the previous proposition implies that the polynomials are also dense in $\mathscr D_{\zeta, k}.$
\end{proof}

We proceed to prove the main result of this section, namely Theorem \ref{thm:Littlewood-Paley-Higher-Order}, which can be thought as a generalized Littlewood-Paley formula for the higher order weighted Dirichlet integral. First we mention a  combinatorial identity and a lemma which will help us in establishing the main result.

\begin{lemma} For $r \in \R$ and for $k, m \in \mathbb Z _{\ge 0}$ with $k\le m$, we have
    \begin{equation*}
    \sum _{l=k}^m \binom{l}{k} r^{m-l}+(1-r) \sum _{l=k+1}^m \binom{l}{k+1} r^{m-l} = \binom{m+1}{k+1}.
\end{equation*}
\label{lem:sum-lem}
\end{lemma}
\begin{proof}
    
The proof of this lemma follows immediately by fixing $k \in \mathbb{Z}_{\ge 0}$ and performing induction on $m \ge k$ and using Pascal's identity.

\end{proof}

 \begin{lemma}\label{lem:H_k-space}
    Let $m, n, k \in \mathbb{Z}_{\ge 0}$ with $k \le m \le n$. For $f(z) = z^m$ and $g(z) = z^n$, we have
    \begin{equation*}
        (-1)^{k-1} \int_{\D} f^{(k)}(z) \overline{g^{(k)}(z)}  H_k(z, \zeta) \, dA(z) =  (1-|\zeta|^2)^{k} \overline{\zeta}^{n-m} \frac{m! (k-1)!}{(m-k)!}.
    \end{equation*}
\end{lemma}

\begin{proof}

    First, we define the weight $w_{k} \left( z \right) = \left( 1-\abs{z}^{2}  \right)^{k-1} \left( 1-\abs{z\zeta}^{2} \right)$ and consider the weighted Bergman space given by 
    \begin{equation*}
        \calA (w_{k}) := \left\{ h \in \operatorname{Hol}\left( \D \right) \, : \, \norm{h}^{2} = \int_{\D} \abs{h\left( z \right)}^{2} w_{k} \left( z \right) dA\left( z \right) < \infty \right\}.
    \end{equation*}

    It is straightforward to verify that $\{z^n: n\in \mathbb Z_{\geqslant 0}\}$ is a family of orthogonal vectors in $ \calA (w_{k})$ and 
    \begin{equation*}
        \norm{z^{n}}_{\calA \left( w_{k} \right)} ^{2} =  \frac{n! (k-1)!}{(n+k)!} \left( 1- \frac{\left( n+1 \right) \abs{\zeta}^{2}}{n+k+1} \right),\qquad n\in \mathbb Z_{\geqslant 0}.
    \end{equation*}

    Now, let $m, n, k \in \mathbb{Z}_{\ge 0}$ with $k \le m \le n$. For $f(z) = z^m$ and $g(z) = z^n$, consider the following:
    \begin{equation}
        \int_{\D} f^{(k)}\left( z \right) \overline{g^{(k)} \left( z \right)} \frac{H_{k} \left( z, \zeta \right)}{ ( 1- \abs{\zeta}^{2} )^{k}} \mathrm{d}A\left( z \right)
         =  \frac{m! n!}{(m-k)! (n-k)!} \int_{\D} \frac{z^{m-k}}{1-\overline{\zeta} z} \overline{\left(  \frac{z^{n-k}}{1-\overline{\zeta} z} \right)} w_{k} \left( z \right) \mathrm{d}A \left( z \right).
         \label{eqn:monomial_H_k-1}
    \end{equation}

    Moreover,
    \begin{align}
        \begin{split}
            \int_{\D} \frac{z^{m-k}}{1-\overline{\zeta} z} \overline{\left(  \frac{z^{n-k}}{1-\overline{\zeta} z} \right)} w_{k} \left( z \right) \mathrm{d}A \left( z \right) 
            &= \ip{\frac{z^{m-k}}{1 -\overline{\zeta} z},\frac{z^{n-k}}{1-\overline{\zeta} z}}_{\calA (w_k)} \\
            &= \ip{ (\overline{\zeta})^{-m+k} \sum_{j \ge m-k} \left( \overline{\zeta}z \right)^j , (\overline{\zeta})^{-n+k} \sum_{l \ge n-k} \left( \overline{\zeta}z \right)^l }_{\calA (w_k)} \\
            &= (\overline{\zeta})^{-m+k} (\zeta)^{-n+k} \sum_{j \ge n-k} \abs{\zeta}^{2j} \frac{j! (k-1)!}{(j+k)!} \left( 1- \frac{\left( j+1 \right) \abs{\zeta}^{2}}{j+k+1} \right) \\
            &=  (\overline{\zeta})^{-m+k} (\zeta)^{-n+k} \left( \abs{\zeta}^{2(n-k)} \frac{(n-k)! (k-1)!}{n!} \right) \\
            &=  \frac{(n-k)! (k-1)!}{n!} (\overline{\zeta})^{n-m}.
        \end{split}
        \label{eqn:monomial_H_k-2}
    \end{align}

    We remark that the fourth equality of \eqref{eqn:monomial_H_k-2} holds because the sum is telescoping and $\zeta \in \D$. The computations in equations \eqref{eqn:monomial_H_k-1} and \eqref{eqn:monomial_H_k-2} complete the proof.
\end{proof}

We can now start the proof of Theorem \ref{thm:Littlewood-Paley-Higher-Order}.  We will use the notation $D _{\zeta, k } \left( \cdot, \cdot \right) $ and $D _{\sigma, k-1 } \left( \cdot, \cdot \right) $ to denote the semi-inner products induced by the semi-norms $\sqrt{D _{\zeta, k } \left( \cdot \right)}$ and  $\sqrt{D _{\sigma, k-1 } \left( \cdot \right)}$ respectively.

\begin{proof}[Proof of Theorem \ref{thm:Littlewood-Paley-Higher-Order}]
        We first claim that for $m, n, k \in \Z_{\ge 0}$ and $\zeta \in \D$,  
    \begin{equation}  
        D_{\zeta, k} (z^m, z^n) = D_{\sigma, k-1} \left( \frac{z^m - \zeta^m}{z-\zeta}, \frac{z^n - \zeta^n}{z-\zeta} \right).  
        \label{eqn:local-douglas-for-monomials}  
    \end{equation}  
    
    We prove this claim by induction on $k$. Without loss of generality, we may assume $k \le m \le n$; otherwise, if either $m$ or $n$ is strictly less than $k$, both sides of \eqref{eqn:local-douglas-for-monomials} vanish, and the result holds trivially. The base case $k = 1$ follows from the Littlewood--Paley Theorem for Hardy spaces (see \cite[Chapter VI, Lemma 3.1]{GarnettBAF}).

    Assume that the theorem holds for some $k-1$, where $k \ge 2$. By induction hypothesis, we have that for all $m, n \in \mathbb Z _{\ge 0}$ satisfying $k-1 \le m \le n$,
    \begin{align}
        \begin{split}
            D_{\zeta, k-1} (z^m, z^n) &= D_{\sigma, k-2} \left( \frac{z^m - \zeta^m}{z-\zeta}, \frac{z^n - \zeta^n}{z-\zeta} \right) \\
            &= D_{\sigma, k-2} \left( \sum_{l=0}^{m-1} \zeta ^{m-1-l} z^{l}, \sum_{l=0}^{n-1} \zeta ^{n-1-l} z^{l} \right) \\
            &= \sum_{l=k-2}^{m-1} \binom{l}{k-2} \abs{\zeta}^{2\left( m-1-l \right)}
        \end{split}
        \label{split:ind-hypo}
    \end{align}
    
    By definition, we have for all $m, n \in \mathbb Z _{\ge 0}$ satisfying $k \le m \le n$,
    \begin{align}
    \begin{split}
        D_{\sigma, k-1} \left( \frac{z^m - \zeta ^m}{z-\zeta}, \frac{z^n - \zeta ^n}{z-\zeta} \right) &= D_{\sigma, k-1} \left( \sum_{l=0}^{m-1} \zeta^{m-1-l} z^{l}, \sum_{l=0}^{n-1} \zeta^{n-1-l} z^{l} \right) \\
        &=  \left( \overline{\zeta} \right)^{n-m} \sum_{l=k-1}^{m-1} \binom{l}{k-1} \abs{\zeta}^{2\left( m-1-l \right)}.   
    \end{split}
    \label{split:eq2}
    \end{align} 

    Now, for $m, n \in \mathbb Z _{\ge 0}$ satisfying $k \le m \le n$, set $f(z) = z^m$ and $g(z) = z^n$ and observe that, by definition,
    \begin{align*}
    \begin{split}
        ( 1-\abs{\zeta}^{2} )^{k} ( ( k-1 )! ) ^{2} D_{\zeta, k} (f, g) &= \int_{\D} f^{\left( k \right)}\left( z\right) \overline{g^{(k)} (z)} U_{\zeta, k} \left( z \right) \, \mathrm{d}A (z)
    \end{split}
    \end{align*}
    and, 
    \begin{align}
        \begin{split}
            \int_{\D} f^{\left( k \right)}\left( z\right) \overline{g^{(k)} (z)} U_{\zeta, k} \left( z \right) \, \mathrm{d}A (z)
            &= (-1)^{k+1} \left( k-1 \right)^{2} \int_{\D} f^{\left( k-1 \right)}\left( z\right) \overline{g^{(k-1)} (z)}   G _{k-1} \left( z, \zeta \right) \, dA(z) \\ &\quad - (-1)^{k+1} \left( k-1 \right) \int_{\D} f^{\left( k-1 \right)}\left( z\right) \overline{g^{(k-1)} (z)}  H_{k-1} \left( z, \zeta \right) \, dA(z).
        \end{split}
        \label{split:Green-f-and-g}
    \end{align}
    We remark that that the equality in \eqref{split:Green-f-and-g} holds because of item $(ii)$ of Proposition \ref{prop:analytic-Gk-properties}, item (4) of Proposition \ref{prop:properties-of-Green-functions} and the polarisation identity.

    Now, observe that 
    \begin{align}
        \begin{split}
            & (-1)^{k+1} \left( k-1 \right)^{2} \int_{\D} f^{\left( k-1 \right)}\left( z\right) \overline{g^{(k-1)} (z)}   G _{k-1} \left( z, \zeta \right) \, dA(z) \\
            &\qquad= - ((k-1)!)^{2} (1-\abs{\zeta}^2)^{k-1} D_{\zeta, k-1} (f,g) \\
            &\qquad= - ((k-1)!)^2 (1-\abs{\zeta}^2)^{k-1} (\overline{\zeta})^{n-m} \sum_{l=k-2}^{m-1} \binom{l}{k-2} \abs{\zeta}^{2\left( m-1-l \right)}
        \end{split}
        \label{split:eq1}
    \end{align}
    where the last equality holds because of equation \eqref{split:ind-hypo}.
Furthermore, we have from Lemma \ref{lem:H_k-space} that

    \begin{align}
        \begin{split}
             (-1)^{k+1} \left( k-1 \right) \int_{\D} f^{\left( k-1 \right)}\left( z\right) \overline{g^{(k-1)} (z)}  H_{k-1} \left( z, \zeta \right) \, dA(z) 
            \\
            = -(1-\abs{\zeta}^2)^{k} (\overline{\zeta})^{n-m} \frac{m! (k-1)!}{(m-k+1)!}
        \end{split}
        \label{split:eq3}
    \end{align}

    Since we wish to show that \eqref{eqn:local-douglas-for-monomials} holds, it suffices to show that 
    \begin{align*}
    \begin{split}
               (-1)^{k+1} \left( k-1 \right)^{2} \int_{\D} f^{\left( k-1 \right)}\left( z\right) \overline{g^{(k-1)} (z)}   G _{k-1} \left( z, \zeta \right) \, dA(z)
               \\ - (1-\abs{\zeta}^2)^{k} ((k-1)!)^{2} D_{\sigma, k-1} (f,g) \\
         = (-1)^{k+1} \left( k-1 \right) \int_{\D} f^{\left( k-1 \right)}\left( z\right) \overline{g^{(k-1)} (z)}  H_{k-1} \left( z, \zeta \right) \, dA(z).
    \end{split}
    \end{align*}

    To prove this, observe that 
    \begin{align*}
    \begin{split}
         & (-1)^{k+1} \left( k-1 \right)^{2} \int_{\D} f^{\left( k-1 \right)}\left( z\right) \overline{g^{(k-1)} (z)}   G _{k-1} \left( z, \zeta \right) \, dA(z) \\
        & \hspace{40mm} - (1-\abs{\zeta}^2)^{k} ((k-1)!)^{2} D_{\sigma, k-1} (f,g) \\
        & = - ((k-1)!)^2 (1-\abs{\zeta}^2)^{k-1} \left( \overline{\zeta} \right)^{n-m} \sum_{l=k-2}^{m-1} \binom{l}{k-2} \abs{\zeta}^{2\left( m-1-l \right)}  \\ & \hspace{15mm}  - (1-\abs{\zeta}^2)^{k} ((k-1)!)^{2} \left( \overline{\zeta} \right)^{n-m} \sum_{l=k-1}^{m-1} \binom{l}{k-1} \abs{\zeta}^{2\left( m-1-l \right)} \\
        & = -((k-1)!)^{2} (1-\abs{\zeta}^2)^{k-1} \left( \overline{\zeta} \right)^{n-m} \left( 1-\abs{\zeta}^{2} \right) \left( \sum_{l=k-1}^{m-1} \binom{l}{k-1} \abs{\zeta}^{2\left( m-1-l \right)} \right) \\
        & \hspace{30mm} + \sum_{l=k-2}^{m-1} \binom{l}{k-2} \abs{\zeta}^{2\left( m-1-l \right)} \\
        & = -((k-1)!)^{2} (1-\abs{\zeta}^2)^{k-1} \left( \overline{\zeta} \right)^{n-m} \binom{m}{k-1} \\
        &= - \left( 1-\abs{\zeta}^{2} \right)^{k-1} \left( \overline{\zeta} \right)^{n-m} \frac{(k-1)! m!}{(m-k+1)!} \\
        & = (-1)^{k+1} \left( k-1 \right) \int_{\D} f^{\left( k-1 \right)}\left( z\right) \overline{g^{(k-1)} (z)}  H_{k-1} \left( z, \zeta \right) \, dA(z).
    \end{split}
    \end{align*}
    We remark that first equality is by equations \eqref{split:eq1} and \eqref{split:eq2} and the third holds by Lemma \ref{lem:sum-lem}, and the last by \eqref{split:eq3}. This proves our claim \eqref{eqn:local-douglas-for-monomials} for monomials. Since both $D_{\zeta,k}(\cdot,\cdot)$ and $D_{\sigma,k-1}(\cdot,\cdot)$ are sesquilinear forms, it follows that \eqref{eqn:local-douglas-for-monomials} holds for all polynomials. Hence, for any polynomials $p,q$, we have
        \begin{equation*}
        D_{\sigma,k-1}\left(\frac{p-p(\zeta)}{z-\zeta}-\frac{q-q(\zeta)}{z-\zeta}\right) = D_{\zeta,k}(p-q).
        \end{equation*}
        Therefore, the map
            \begin{equation*}
            T_\zeta p:=\frac{p-p(\zeta)}{z-\zeta}
            \end{equation*}
        is an isometry from the polynomials equipped with the seminorm $D_{\zeta,k}^{1/2}$ into the polynomials equipped with the seminorm $D_{\sigma,k-1}^{1/2}$. Since polynomials are dense in $D_{\zeta,k}$ by Proposition \ref{prop:equivalence-of-norm}, the identity extends by continuity to every $f\in D_{\zeta,k}$, and we obtain
            \begin{equation*}
            D_{\zeta,k}(f)=D_{\sigma,k-1}\left(\frac{f-f(\zeta)}{z-\zeta}\right), \qquad f\in D_{\zeta,k}.
            \end{equation*}
        Finally, let $f\in \operatorname{Hol}(\mathbb D)$. If $f\in D_{\zeta,k}$, then the above identity has already been proved. Conversely, if
            \begin{equation*}
            D_{\sigma,k-1}\left(\frac{f-f(\zeta)}{z-\zeta}\right)<\infty,
            \end{equation*}
        then $\frac{f-f(\zeta)}{z-\zeta}\in D_{k-1}$, and hence
            \begin{equation*}
            f(z)=f(\zeta)+(z-\zeta)\frac{f(z)-f(\zeta)}{z-\zeta}\in D_{k-1}.
            \end{equation*}
        By Proposition 3.1, this implies that $f\in D_{\zeta,k}$. Therefore, the identity holds for every $f\in \operatorname{Hol}(\mathbb D)$. This completes the proof.
\end{proof}

A version of Theorem \ref{thm:Littlewood-Paley-Higher-Order} also holds for $\zeta \in \T$, as shown by S. Ghara et al. \cite{GGRLDF}:

\begin{theorem}[S. Ghara et al., \cite{GGRLDF}]
    Let $k \ge 1$, $\zeta \in \T$, and $f \in \operatorname{Hol} (\D)$. Then $f \in \mathscr D_{\zeta, k}$ if and only if $f = \alpha + (z-\zeta) g$ for some $g \in \operatorname{Hol} (\D)$ with $D_{\sigma, k-1} (f) < + \infty$ and $\alpha \in \C$. Moreover, in this case, the following statements hold:
    \begin{enumerate}[label=(\roman*)]
        \item $D_{\zeta, k} (f) = D_{\sigma, k-1} (g)$,
        \item $f(z) \to \alpha$ as $z \to \zeta$ in each oricyclic approach region $\abs{z-\zeta} < \kappa (1-\abs{z}^2)^{\frac{1}{2}}$, $\kappa > 0$. In particular, $f^{*} (\lambda)$ exists and is equal to $\alpha$.
    \end{enumerate}
\end{theorem}
As a consequence, we have that for $f \in H^2$, $\zeta \in \overline{\D}$, and $k \ge 1$,  
\begin{equation}  
    D_{\zeta, k} (f) = D_{\sigma, k-1} \left( \frac{f(z) - f^{*} (\zeta)}{z-\zeta} \right)
    \label{eqn:local-douglas-formula}
\end{equation}  
where $f^{*} (\zeta) = f(\zeta)$ if $\zeta \in \D$; else, if $\zeta \in \T$, $f^{*} (\zeta)$ is the radial limit provided that it exists.  

\section{Higher Order weighted Dirichlet Spaces \texorpdfstring{$\mathscr D_{\mu,k}$}{D(mu,k)} with \texorpdfstring{$\mu$}{mu} supported on \texorpdfstring{$\overline{\mathbb D}$}{bar D}}

Let $k \in \mathbb Z _{\ge 1}$, and let $\mu\in M_+(\overline{\D}).$ We write $\mu = \mu_1+\mu_2$ where $\mu_1= \mu|_{\D}$ and $\mu_2= \mu|_{\mathbb T}.$ So for each Borel measurable subset $A \subset \C$, we have 
\begin{equation*}
    \mu (A) = \mu_1 (A \cap \D) + \mu_2 (A \cap \T).
\end{equation*} 
We consider the function $U_{\mu, k} : \mathbb D \to \mathbb C$ given by 
\begin{align*}
    U_{\mu,k}(z) &:= \frac{1}{((k-1)!)^{2}} 
    \int_{\D} \frac{(-1)^{k+1} G_k(z,\zeta)}{(1-|\zeta|^{2})^{k}}\, d\mu_1(\zeta) 
    + \frac{1}{k!(k-1)!} \int_{\T} \frac{(1-|z|^{2})^{k}}{|z-\zeta|^{2}}\, d\mu_2(\zeta),\quad z\in \D,
\end{align*}

Observe that if $\zeta \in \D$, and $\mu=\delta_{\zeta},$  then  $U_{\delta_{\zeta}, k} =U_{\zeta, k}$ as given by \eqref{eqn:local-weight-in-open-disc}. Since $U_{\zeta, k}$ is strictly positive for every $\zeta\in \overline{\D},$ it follows that $U_{\mu, k}(z)\geqslant 0$ for every $z\in \D.$ Note that $U_{\mu,k}= U_{\mu_{1},k} +U_{\mu_2,k},$ where $U_{\mu_2,k}(z)=\frac{1}{k!(k-1)!} (1-|z|^2)^{k-1} P_{\mu_2}(z)$ for $z\in \D.$  Since $P_{\mu_2},$ the Poisson integral of the measure $\mu_2,$ is a harmonic function on $\mathbb D$, the function $U_{\mu_2,k}$ is a $k$-harmonic function on $\D$, that is, $\Delta ^{k} U_{\mu_2, k}= 0.$ It also follows from item $(f)$ of Proposition \ref{Basic properties of Green's function} that $ \left ( -\Delta \right )^{k} U_{\mu_{1}, k}(z) = \frac{d\mu_1(z)}{(1-|z|^2)^k} \geqslant 0$ (in the sense of distributions), hence, $U_{\mu,k}$ is a $k$-superharmonic function. 

\begin{definition}[Higher order weighted Dirichlet type integrals]
Let $k \in \mathbb Z_{\ge 1}$, and let $\mu\in M_+(\overline{\D}).$ The \textit{$k$th order weighted Dirichlet integral} and the associated \textit{$k$th order weighted Dirichlet type space} are defined respectively by
\begin{align*}
\begin{split}
D_{\mu, k} (f) &:= \int_{\D} \lvert f^{(k)} (z) \rvert^2 U_{\mu, k} (z) \, dA(z), \qquad f \in \operatorname{Hol} (\D), \\
\mathscr D_{\mu, k} &:= \left \{ f \in \operatorname{Hol} (\D) \, : \, D_{\mu, k} (f) < \infty \right \}.
\end{split}
\end{align*}
\end{definition}
Let $\mu\in M_+(\overline{\D}).$ For $f\in \operatorname{Hol} (\D)$, we define the $0$-th order weighted Dirichlet integral $D_{\mu,0}(f)$  
\begin{align*}
    D_{\mu,0}(f) := \int_{\mathbb D} |f(z)|^2 \, d\mu_1(z) + \lim\limits_{r \to 1^{-}}
        \int_{\mathbb T} |f(r\zeta)|^2 P_{\mu_2}(r\zeta) \,d\sigma(\zeta),
\end{align*}
whenever the latter limit exists. For the normalized  Lebesgue measure $\sigma$ on the circle $\T,$ it is straightforward to see that $D_{\sigma,0}(f)= \|f\|_{H^2}^2$ and we define the space $\mathscr D_{\sigma,0}= H^2,$ the Hardy space.  We remark that if $\mu$ is a measure supported on $\overline{\D}$, $f\in \operatorname{Hol} (\D)$ and $k \ge 1$, then, 
\begin{equation} \label{eqn:local-dirichlet-integral}
    D_{\mu, k} (f) = \int_{\overline{\D}} D_{\zeta, k} (f) \, d\mu (\zeta).
\end{equation}

For $0<r<1,$ the $r$-dilation $f_r$ of a function $f\in  \operatorname{Hol} (\D),$ is given by $f_r(z)=f(rz)$ for $z\in \D.$  We show that $f_r$ converges to the function $f$ as $r\to 1^{-}. $ in $\mathscr D_{\mu, k}.$ In this regard, the following proposition extends \cite[Lemma 4.1]{1993multiplication}, \cite[Corollary 1.5]{MR2019HM}.
\begin{proposition} \label{prop:dilations}
    Let $k\in \mathbb Z_{\geqslant 1},$ $\mu\in M_+(\overline{\D}),$  and $f\in \mathscr D_{\mu, k}.$ For $r \in (0,1),$ we have that 
    \begin{equation*}
        D_{\mu, k} (f_r) \le 2 D_{\mu, k} (f),
    \end{equation*}
    and consequently,  $D_{\mu, k} (f_r - f) \to 0 $  as $r \to 1^{-}.$
\end{proposition}
    
\begin{proof}
    We only deal with $k\ge 2$ as the case $k=1$ can be found in \cite[Corollary 1.5]{MR2019HM}. Write $\mu=\mu_1+\mu_2$ where $\mu_1 = \mu |_{\D}$ and $\mu_2 = \mu |_{\T}$. From Corollary \ref{cor:f_r}, we have that for any $\zeta \in \D$, $D_{\zeta, k} (f_r) \le 2 D_{\zeta, k} (f)$. Integrating both sides with respect to the measure $\mu_1$, we have  
    \begin{align*}
        D_{\mu_{1}, k} (f_r) \le 2 D_{\mu_{1}, k} (f).
    \end{align*}
    Combining this together with \cite[Theorem 1.1]{GGR22}, we have $D_{\mu, k} (f_r) \le 2 D_{\mu, k} (f).$

For the second part, the result is already known in the case $\mu=\mu_2$; see \cite[Theorem 1.1]{GGR22}. It therefore suffices to treat the case $\mu=\mu_1.$ We first note that $D_{\sigma,k-1}(f_r - f) \to 0$ as $r \to 1^{-}$ for any $f \in \mathscr{D}_{k-1}$. By Proposition~\ref{eqn:equivalent-norms-on-Dk}, it follows that $
D_{\zeta,k}(f_r - f) \to 0 \quad \text{as } r \to 1^{-} $
for every $f \in \mathscr{D}_{\zeta,k}$ and $\zeta \in \mathbb{D}.$ On the other hand, by the first part, for $\zeta \in \mathbb{D}$ we have
\begin{align*}
 D_{\zeta,k}(f_r - f)
\leqslant 2 D_{\zeta,k}(f_r) + 2 D_{\zeta,k}(f)
\leqslant 6 D_{\zeta,k}(f), \qquad f \in \mathscr{D}_{\zeta,k}.   
\end{align*}
In view of \eqref{eqn:local-dirichlet-integral}, the family $\{D_{\zeta,k}(f_r - f) : 0 < r < 1\}$ is dominated by an integrable function. Hence, the dominated convergence theorem yields
\begin{align*}
    D_{\mu_1,k}(f_r - f) \to 0 \quad \text{as } r \to 1^{-}.
\end{align*}
This completes the proof.  
\end{proof}
 We remark that the bound of $2$ in Proposition \ref{prop:dilations} is in no way sharp. In fact, a proof of induction using Littlewood-Paley formula and Local Douglas formula \eqref{eqn:local-douglas-formula} shows that the bound can be improved to $1$. We do not present the details here and refer the readers \cite[Theorem 2.6]{GGRLDF} for the proof. 

We note down the following corollary as this will be used to derive certain properties of the space $\mathscr D_{\mu, k}.$
 
\begin{corollary}
    Let $k\ge 1$ and $\mu\in M_+(\overline{\D}).$  For each $f\in \mathscr D_{\mu,k}$ and $\varepsilon >0 $, there exists an analytic polynomial $p\in\mathbb C[z]$ such that $D_{\mu, k} (f - p) < \varepsilon$.
    \label{cor:density-of-poly}
\end{corollary}
\begin{proof}
Let $f\in \mathscr D_{\mu,k}.$ By Proposition \ref{prop:dilations}, there exists $t\in (0,1)$ such that  $D_{\mu, k} (f - f_t) < \frac{\varepsilon}{2}.$ Fix this $t,$ and consider $s_n(f_t),$ the $n$-th degree Taylor polynomial of $f_t$ around zero. Since $f_t$ is holomorphic in a neighborhood of $\overline{\D},$ for each $j\in \mathbb Z_{\geqslant 0},$ the polynomials $s_n(f_t^{(j)})$ converges uniformly to $f_t^{(j)}$ in a neighborhood of $\overline{\D}$ as $n\to \infty.$ It follows that $D_{\mu, k} (f_t - s_n(f_t)) < \frac{\varepsilon}{2}$ for large $n.$ Hence we obtain that $D_{\mu, k} (f - s_n(f_t)) < \varepsilon$ for large $n.$
\end{proof}

The following relates the weighted Dirichlet integral $D_{\mu,1}(f)$ to the Dirichlet integral $D_{\sigma,2}(f)$. This is a natural generalization of \cite[Proposition 5.3]{Rydhe19}.

\begin{proposition} \label{Carleson measure}
    Let  $\mu\in M_+(\overline{\D}).$ Then, the measure $\nu$ given by $d\nu (z) = U_{\mu,1} (z)dA(z)$ determines a Carleson measure for $H^2,$ that is, there exist a scalar $C_{\mu}>0$ such that
    \begin{align*}
        \int_{\D}|g(z)|^2 U_{\mu,1} (z)dA(z) \leqslant C_{\mu} \|g\|_{H^2}^2,\qquad g\in\operatorname{Hol}(\D).
    \end{align*}
    Consequently, we have $D_{\mu,1}(f) \leqslant 4C_{\mu} \left(|f'(0)|^2 + D_{\sigma,2}(f)\right),$ for every $f\in\operatorname{Hol}(\D).$
\end{proposition}

\begin{proof}
    In view of \cite[Chapter VI, Lemma 3.3]{MR2261424}, we need to show that 
\begin{equation*}
    \sup_{w\in \D} \int_{\D} \frac{1-\abs{w}^2}{\abs{1-\overline{w}z}^2}\, d\nu (z) = \sup_{w\in \D} \int_{\D} \frac{1-\abs{w}^2}{\abs{1-\overline{w}z}^2}\, U_{\mu,1} (z) dA (z) < +\infty.
\end{equation*}

Fix $w\in \D$. Consider the function $F_{w} \in\operatorname{Hol}(\D),$ given by 
\begin{equation*}
    F_w (z) = \begin{cases}
        \frac{1}{\overline{w}} \log \frac{1}{1-\overline{w}z}, & w\ne 0, \\
        z, & w=0.
    \end{cases}
\end{equation*}

It is easy to check that $F_w ' (z) = \frac{1}{1-\overline{w}z}$ for each $z \in \D$. Consequently, 
\begin{equation*}
    D_{\mu, 1} (F_w) = \int_{\D} \frac{1}{\abs{1-\overline{w}z}^2} U_{\mu,1} (z) \, dA(z).
\end{equation*}

We show that for each $\zeta\in \overline{\D}, w \in \D$, we have
\begin{equation}
    D_{\zeta,1} (F_w) \le \frac{4}{1-\abs{w}^2}.
    \label{eq:D-zeta-F-w}
\end{equation}

We will be done, for then, 
\begin{align*}
    \sup_{w\in \D} \int_{\D} \frac{1-\abs{w}^2}{\abs{1-\overline{w}z}^2}\, U_{\mu} (z) dA (z) &= \sup_{w \in \D} (1-\abs{w}^2) D_{\mu, 1} (F_w) \\
    &= \sup_{w\in \D} (1-\abs{w}^2) \int_{\overline{\D}} D_{\zeta, 1} (F_w) \, d\mu (\zeta) \\
    &\le 4\mu(\overline{\D}) < \infty.
\end{align*}

We proceed to show inequality \eqref{eq:D-zeta-F-w}. To this end, first note that, if we take $w=0,$ then $D_{\zeta,1}(F_w)= 1$ for every $\zeta\in \overline{\D},$ and hence \eqref{eq:D-zeta-F-w} is established for $w=0.$  Now take $\zeta\in \overline{\D}, w \in \D \setminus \{0\}$ and notice that
\begin{equation*}
    F_w (z) = \sum_{m=1}^{\infty} \frac{\overline{w}^{m-1}}{m} z^m, \qquad z\in \D,
\end{equation*}
and for $z\in \D$,
\begin{align*}
    \frac{F_w (z)-F_w (\zeta)}{z-\zeta} &= \sum_{m=1}^{\infty} \frac{\overline{w}^{m-1}}{m}\, \frac{z^m - \zeta^m}{z-\zeta} \\
    &= \sum_{m=1}^{\infty} \frac{\overline{w}^{m-1}}{m}\, \left( \sum_{j=0}^{m-1} z^{j} \zeta^{m-1-j}\right) \\
    &= \sum_{j=0}^{\infty} \left( \sum_{m=j+1}^{\infty} \frac{\overline{w}^{m-1}}{m} \zeta ^{m-1-j} \right) z^j.
\end{align*}

Consequently, we have by the local Douglas formula in \eqref{eqn:local-douglas-formula} that,
\begin{equation*}
    D_{\zeta,1}(F_w) =\sum_{j=0}^{\infty} \abs{\sum_{j=0}^{\infty} \frac{\overline{w}^{m-1}}{m} \zeta^{m-1-j}}^2 \le \sum_{j=0}^{\infty} \left( \sum_{m=j+1}^{\infty} \frac{\abs{w}^{m-1}}{m} \right)^2
\end{equation*}
Set $\varrho = |w|$. We consider the sum
\begin{equation*}
    S(\varrho) := \sum_{j=0}^{\infty} \left( \sum_{m=j+1}^{\infty} \frac{\varrho^{m-1}}{m} \right)^2.
\end{equation*}
Using the equality
\begin{equation*}
    \sum_{m=j+1}^{\infty} \frac{\varrho^{m-1}}{m} = \frac{1}{\varrho} \int_{0}^{\varrho} \frac{x^j}{1-x} \, dx,
\end{equation*}
we have
\begin{align*}
    S(\varrho) &= \frac{1}{\varrho^2} \sum_{j=0}^{\infty} \int_{0}^{\varrho} \int_{0}^{\varrho} \frac{(xy)^j}{(1-x)(1-y)} \, dx\, dy\\
    &= \frac{1}{\varrho^2} \int_{0}^{\varrho} \int_{0}^{\varrho} \frac{1}{(1-x)(1-y)(1-xy)} \, dx\, dy. \\
    &= \frac{(1+\varrho) \ln(1+\varrho)}{\varrho^2(1-\varrho)} + \frac{\ln(1-\varrho)}{ \varrho ^2}\\
    &= \frac{2 \ln (1+\varrho) }{\varrho(1-\varrho)} +  \frac{\ln(1-\varrho^2)}{ \varrho ^2}.
\end{align*}
Using the inequalities $\ln (1+x) \le x$ for $x \ge 0$ and $\ln (1- x^2) \le 0$ for $x \in (0,1)$, we have that 
\begin{equation*}
    S (\varrho) \le \frac{2}{1-\varrho} = \frac{2 (1+\varrho)}{1-\varrho^2} \le \frac{4}{1-\varrho^2}.
\end{equation*}
Consequently, we have
\begin{equation*}
    D_{\zeta, 1} (F_w) \le \frac{4}{1-\abs{w}^2}.
\end{equation*}
This completes the proof of the first part of the Proposition. Consequently we have that
\begin{align*}
    D_{\mu,1}(f) = \int_{\D}|f'(z)|^2 U_{\mu,1} (z)dA(z) \leqslant C_{\mu} \|f'\|_{H^2}^2 \leqslant 4C_{\mu} \left ( |f'(0)|^2 + D_{\sigma,2}(f) \right),
\end{align*}
 for every $f\in\operatorname{Hol}(\D).$
\end{proof}

Now we establish a difference formula which will be essential in studying the properties of the operator $M_z$ of multiplication by the co-ordinate function on $ \mathscr D_{\mu, k}.$ The following formula was already established in \cite[Proposition 2.7]{GGR22} for the special case when $\mu\in  M_+(\T),$ see also \cite[Proposition 3.4]{Rydhe19}.
\begin{proposition}\label{prop:diff-formula}
    Let $\mu\in M_+(\overline{\D}),$ $k\in \mathbb Z_{\geqslant 1}$ and  $f\in \mathscr D_{\mu, k}.$ Then following holds:
    \begin{equation*}
        D_{\mu, k} (zf) = D_{\mu, k} (f) + D_{\mu, k-1} (f).
    \end{equation*}
\end{proposition}

\begin{proof}
      First let's consider the case of $k=1.$ In view of \cite[Theorem 4.1]{Richter91} and  \cite[Ch. IV, Proposition 1.6]{1993multiplication}, it follows that 
\begin{align*}
   D_{\mu,1}(zf)= D_{\mu,1}(f) +  D_{\mu,0}(f),\quad f\in\mathscr{D}_{\mu, 1},
\end{align*}
see also \cite[Ch. IV, Theorem 1.10]{1993multiplication}. 

Now, we consider the case $k \ge 2$. Write $\mu=\mu_1+\mu_2$ where $\mu_1 = \mu |_{\D}$ and $\mu_2 = \mu |_{\T}$. The result is already established in \cite[Proposition 2.7]{GGR22} when $\mu=\mu_2.$ Therefore, in view of the equation \eqref{eqn:local-dirichlet-integral}, it is sufficient to prove the result for the case of $\mu=\delta_{\zeta}$ for $\zeta\in\D.$  Consider $\zeta \in \D$ and  $D_{\zeta,k}(f) < \infty.$ Then we can write
\begin{equation*}
    \frac{zf - (zf) (\zeta)}{z-\zeta} = z \frac{f- f (\zeta)}{z-\zeta} + f(\zeta).
\end{equation*}
Using equation \eqref{eqn:local-douglas-formula}, it follows that
\begin{align*}
    \begin{split}
        D_{\zeta,k}(zf) &=  D_{\sigma, k-1} \left( z \frac{f-f\left( \zeta \right)}{z-\zeta} \right) \\
        &= D_{\sigma, k-1} \left( \frac{f-f\left( \zeta \right)}{z-\zeta} \right) + D_{\sigma, k-2} \left( \frac{f-f\left( \zeta \right)}{z-\zeta} \right) \\
        &= D_{\zeta,k}(f) + D_{\zeta,k-1}(f).
    \end{split}
    \end{align*}
    Note that the second equality follows from the difference formula in   \cite[Proposition 2.7]{GGR22}. This completes the proof.
\end{proof}
\begin{remark}\label{zf vs f}
Let $\mu\in M_+(\overline{\D}),$ $k\in \mathbb Z_{\geqslant 1}.$ The above proof also shows that if $D_{\mu,k}(zf) < \infty$ for some $f\in H^2,$ then $D_{\mu,k}(f) < \infty$ and $D_{\mu,k-1}(f) < \infty.$ Moreover, in this case we have $D_{\mu,k}(zf) \geqslant D_{\mu,k}(f).$ 
\end{remark}

Let $L: \operatorname{Hol} (\D) \to \operatorname{Hol} (\D)$ be the backward shift operator given by
\begin{equation*}
    (Lf)(z) := \frac{f(z)-f(0)}{z}, \qquad z \in \D .
\end{equation*}
In the following lemma we show that $L$ maps $\mathscr D_{\mu, k}$ into itself, that is, $\mathscr D_{\mu,k}$ is invariant under backward shift.
\begin{lemma}\label{prop:closed-under-backward-shift}
	Let $k\in \mathbb Z_{\geqslant 1}$, $\mu\in M_+(\overline{\D}),$ and $f \in\mathscr D_{\mu,k}.$  Then
\begin{itemize}
    \item[\rm (i)] $D_{\mu, k} \left( Lf \right) \le D_{\mu, k} \left( f \right).$
    \item[\rm (ii)] $D_{\mu, k} \left( L^{j}f \right) \to 0$ as $j\to \infty$
\end{itemize}
    \end{lemma}
    \begin{proof}
	 (i) Write $g = Lf$. Then $zg(z) = f(z)-f\left( 0 \right)$. Since $D_{\mu, k} \left( f \right)= D_{\mu, k} \left( zg \right),$ by Remark \ref{zf vs f}, it follows that  $D_{\mu, k} \left( f \right) \ge D_{\mu, k} \left( g \right) =D_{\mu,k } \left( Lf \right).$\\
     
     (ii) Let $\epsilon >0.$ By Corollary \ref{cor:density-of-poly}, there exists an analytic polynomial $p \in \C \left[ z \right]$ such that $D_{\mu,k} \left( f-p \right) < \varepsilon$. Using the part (i) result, we obtain that for $j \ge $ degree $(p)+1$,
	\begin{equation*}
	    D_{\mu, k} \left( L^{j}f \right) = D_{\mu, k} \left( L^{j}\left( f-p \right) \right) \le D_{\mu, k} \left( f-p \right) + D_{\mu, k} \left( L^{j}p \right) < \varepsilon.
	\end{equation*}
	Consequently,  we have that $D_{\mu, k} \left( L^{j}f \right) \to 0$ as $j\to \infty.$
    \end{proof}

    \begin{proposition}\label{one step up formula for measure}
    Let $k\in \mathbb Z_{\geqslant 0},$  and $\mu\in M_+(\overline{\D}).$ Then for any $f \in\mathscr D_{\mu,k+1},$ we have that 
    \begin{equation*}
	\sum_{j=1}^{\infty} D_{\mu, k} \left( L^{j}f \right) = D_{\mu, k+1} (f).
    \end{equation*}
    \end{proposition}
    
\begin{proof}
Write $\mu=\mu_1+\mu_2$ where $\mu_1 = \mu |_{\D}$ and $\mu_2 = \mu |_{\T}$. The result is already established in \cite[Lemma 2.10 and Corollary 3.3]{GGR22} when $\mu=\mu_2.$ Therefore, in view of the equation \eqref{eqn:local-dirichlet-integral}, it is sufficient to prove the result for the case of $\mu=\delta_{\zeta}$ for $\zeta\in\D.$ First we consider the case $k=0$.  Consider a function $f\in \mathscr D_{\zeta,1}$ and assume that the power series representation of $f$ is given by $f\left( z \right) = \sum_{m \ge 0} a_m z^m$. For $z, \zeta\in\mathbb D,$ we have
    \begin{align*}
	\frac{f\left( z \right) - f\left( \zeta \right)}{z-\zeta} 
	&= \sum_{m\ge 1} \left( a_{m} \sum_{j=0}^{m-1} z^{m-1-j} \zeta ^{j} \right) \\
	&=\sum_{m\ge 1} \left( \sum_{j\ge m} a_{j} \zeta^{j-m} \right) z^{m-1}  \\
	&= \sum_{m\ge 1} \left(\left( L^{m}f \right) \left( \zeta \right)\right) z^{m-1}.
    \end{align*}
    Consequently, we have 
    \begin{align*}
	D_{\zeta, 1} \left( f \right) &= \norm{\frac{f-f\left( \zeta \right)}{z-\zeta}}_{H^{2}}^{2} = \sum_{m\ge 1} \abs{\left( L^{m}f \right) \left( \zeta \right)}^{2} \\
    &= \sum_{m\ge 1} \abs{\left( L^{m}f \right) \left( \zeta \right)}^{2}
    = \sum_{m\ge 1} D_{\zeta, 0} \left( L^{m}f \right).
    \end{align*}
   This establishes the identity when $k=0.$ 
   
   Now consider the case of $k\ge1$. Let $\mu\in M_+(\overline{\D})$ and $f\in \mathscr D_{\mu,k+1}.$ Since $zL(f) = f-f\left( 0 \right),$ it follows that $D_{\mu, k+1} \left( zLf \right)= D_{\mu,k+1} \left( f \right).$ By the difference formula in Proposition \ref{prop:diff-formula} and Remark \ref{zf vs f}, we have that $Lf\in \mathscr D_{\mu,k+1} \cap  \mathscr D_{\mu,k},$ and 
   $D_{\mu, k+1} \left( f \right) = D_{\mu,k+1} \left( Lf \right) + D_{\mu, k} \left( Lf \right)$ for any $f \in \mathscr D _{\mu, k+1}$. By induction, we have that for each $m \ge 1$,
    \begin{equation*}
	D_{\mu, k+1} \left( f \right) = D_{\mu, k+1} \left( L^{m}f \right) + \sum_{n=1}^{m} D_{\mu, k} \left( L^{n}f \right).
    \end{equation*}
The required identity now follows from part $\rm(ii)$ of Lemma \ref{prop:closed-under-backward-shift}.
\end{proof}

In the following two propositions, we explore the embedding relationship between the weighted Dirichlet-type space $\mathscr{D}_{\mu, k}$ and $\mathscr{D}_{\sigma,j}= \mathscr D_j.$ These results were already established in \cite[Proposition 5.5, Proposition 5.7]{Rydhe19} in the special case where the measure $\mu$ is supported on the unit circle $\mathbb T.$ 

Before proceeding, we record an auxiliary lemma that will be instrumental in proving the next two propositions. The proof of the following lemma follows from routine computations combined with an application of the higher-order Littlewood–Paley formula in Theorem  \ref{thm:Littlewood-Paley-Higher-Order}.

\begin{lemma}
    Let $\zeta\in \D$, $n \in \Z_{\ge 0}.$ Let $f \in \operatorname{Hol} (\D)$ and $f(z)= \sum_{j\geqslant 0} \hat{f}(j) z^j,$ for $z\in\D.$ Define 
    \begin{equation*}
        g_{\zeta} (z):= \frac{f(z) - f(\zeta)}{z-\zeta} = \sum_{j =0}^{\infty} \hat{g}(j) z^j, \qquad z \in \D.
    \end{equation*}
    Then 
    \begin{equation*}
        D_{\zeta, n+1} (f) = \sum\limits_{j= n}^{\infty} \binom{j}{n} \abs{\hat{g}(j)}^{2},
    \end{equation*}
    where 
    \begin{equation*}
        \hat{g}(j) = \sum\limits_{m=j+1}^{\infty} \hat{f} (m) \zeta^{m-1-j},\qquad j\in \mathbb Z_{\geqslant 0},
    \end{equation*}  
    \begin{equation*}
        \hat{f} (m) = \hat{g}(m-1)-\zeta \hat{g}(m), \qquad m \in \Z_{\ge 1}.
    \end{equation*}
    \label{lem:local-dirichlet-integral-in-terms-of-coefficients}
\end{lemma}

\begin{proposition}\label{prop:sigma-vs-mu}
    Let $\mu \in M_{+} (\overline{\D})\setminus \{0\}.$ Then following holds :
    \begin{itemize}
        \item[(i)] If $k \in \Z_{\ge 0},$ then
        \begin{equation*}
            D_{\sigma, k} (f) \lesssim \abs{\hat{f} (k)}^2 + D_{\mu, k+1} (f), \qquad f\in \operatorname{Hol}(\D).
        \end{equation*}
        \item[(ii)] If $k\in \mathbb Z_{\geqslant 1},$ then
        \begin{equation*}
         D_{\mu,k}(f) \lesssim D_{\sigma,k}(f) + D_{\sigma, k+1}(f), \qquad f\in \operatorname{Hol}(\D).
        \end{equation*}
    \end{itemize}
\end{proposition}

\begin{proof}\begin{itemize}
    \item[(i)] We decompose $\mu=\mu_1+\mu_2,$ where $\mu_{1}= \mu|_{\D}$ and $\mu_2 = \mu|_{\T}$. If $\mu_1=0,$ the proof is complete by \cite[Proposition 5.5 (i)]{Rydhe19}. Therefore, we may assume $\mu_2 =0,$ without loss of generality. Now, let $f(z) = \sum_{j \ge 0} \hat{f}(j) z^j.$ Then we have 
    \begin{align*}
        D_{\sigma, k} (f)
        &= \abs{\hat{f}(k)}^2 + \sum\limits_{j=k+1}^{\infty} \binom{j}{k} \abs{\hat{f} (j)}^2 \\
        &\le \abs{\hat{f}(k)}^2 + 2 \sum_{j=k+1}^{\infty} \binom{j}{k} \left( \abs{\hat{g}(j)}^2 + \abs{\hat{g}(j-1)}^2 \right),
    \end{align*}
    where $\hat{g}(j)$ is defined as in Lemma \ref{lem:local-dirichlet-integral-in-terms-of-coefficients}. Since $\binom{j+1}{k}\le (k+1) \binom{j}{k}$, for $j\geqslant k,$ we find that
    \begin{align*}
        D_{\sigma, k} (f) &\le \abs{\hat{f} (k)}^2  + 2 \sum\limits_{j=k+1}^{\infty} \binom{j}{k} \abs{\hat{g}(j)}^2 + 2 (k+1) \sum\limits_{j=k}^{\infty} \binom{j}{k} \abs{\hat{g}(j)}^2\\
        &\le \abs{\hat{f(k)}}^2 + 2(k+2) D_{\zeta, k+1} (f),
    \end{align*}
    where the last inequality can be inferred from Lemma \ref{lem:local-dirichlet-integral-in-terms-of-coefficients}. Integrating both sides with respect to the measure $\mu$, we obtain 
    \begin{equation*}
        D_{\sigma, k} (f)  \le \abs{\hat{f}(k)}^2 + \frac{2 (k+2)}{\mu (\D)} D_{\mu, k+1} (f),
    \end{equation*}
    which is exactly what we wanted to show.
    
    \item[(ii)] We prove the result by induction on $k.$ The case of $k=1$ is follows from Proposition \ref{Carleson measure}. Assume that the result holds true for $k=1,2,\ldots,n,$ for some $n\geqslant 1.$ Therefore there exists a $C>0$ such that $D_{\mu,n}(f) \leqslant C\left( D_{\sigma,n}(f) + D_{\sigma, n+1}(f)\right),$ for every $f\in \operatorname{Hol}(\D).$  In particular, we have
    \begin{align*}
        D_{\mu,n}(L^jf) \leqslant C\left( D_{\sigma,n}(L^jf) + D_{\sigma, n+1}(L^jf)\right),
    \end{align*} for each $j\in \mathbb Z_{\geqslant 1}$ and $f\in \operatorname{Hol}(\D).$ In view of Proposition \ref{one step up formula for measure}, it then follows that 
    \begin{align*}
        D_{\mu,n+1}(f) \leqslant C \left ( D_{\sigma,n+1}(f) + D_{\sigma, n+2}(f)\right).
    \end{align*} This completes the induction step for $k=n+1.$
\end{itemize}  
\end{proof}

\begin{proposition} \label{prop:boundedness-of-mu}
    Let $\mu \in M_{+} (\overline{\D})\setminus \{0\},$ and $k \in \Z_{\ge 1}$. Then 
    \begin{equation*}
        D_{\mu, k} (f) \le D_{\mu, k} (zf) \lesssim \abs{\hat{f} (k-1)}^2 + D_{\mu, k} (f),\qquad f\in \operatorname{Hol}(\D).
    \end{equation*}
    Consequently, for a function $f\in \operatorname{Hol} (\D),$ we have that $f \in \mathscr{D}_{\mu, k}$ if and only if $zf \in \mathscr{D}_{\mu, k}$. Furthermore, it follows that $\mathscr{D}_{\mu, k} \subset \mathscr{D}_{\mu, k-1}$ for every $k\in \mathbb Z_{\geqslant 2}.$
\end{proposition}
\begin{proof}
     Decompose $\mu=\mu_1+\mu_2,$ where $\mu_{1}= \mu|_{\D}$ and $\mu_2 = \mu|_{\T}$. The result is established in \cite[Proposition 5.7]{Rydhe19} for the case when $\mu=\mu_2 \in M_{+} (\T)\setminus \{0\},$ therefore, it is sufficient to consider the case when $\mu=\mu_1.$ Moreover, note that the result for $k=1$ follows from \cite[IV, Proposition 1.6]{1993multiplication}. Therefore we assume that $k\geqslant 2.$ To this end, let $\zeta \in \D$. Write  $f(z)= \sum_{j\geqslant 0} \hat{f}(j) z^j,$ for $z\in\D.$ Define
    \begin{equation*}
        g(z) := \frac{f(z) - f(\zeta)}{z-\zeta}= \sum_{j =0}^{\infty} \hat{g}(j) z^j, \qquad z\in \D.
    \end{equation*}
    Then
    \begin{equation*}
        \frac{zf - (zf) (\zeta)}{z-\zeta} = zg+f(\zeta), \qquad z\in \D.
    \end{equation*}
    
    Consequently, we have
    \begin{align*}
        D_{\zeta, k} (zf) & = D_{\sigma, k-1} (zg)  = \abs{\hat{g}(k-2)}^2 + \sum\limits_{j=k-1}^{\infty} \binom{j+1}{k-1} 
        \abs{\hat{g} (j)}^2,\\
        & \leqslant \abs{\hat{g}(k-2)}^2 + k \sum\limits_{j=k-1}^{\infty} \binom{j}{k-1} \abs{\hat{g} (j)}^2\\
        &= \abs{\hat{g}(k-2)}^2 + k D_{\sigma, k-1} (g).
    \end{align*}
By Lemma \ref{lem:local-dirichlet-integral-in-terms-of-coefficients}, we have
    \begin{equation*}
        \hat{g} (k-2) = \hat{f} (k-1) + \zeta\hat{g}(k-1)
    \end{equation*}
    and consequently,
    \begin{equation*}
        \abs{\hat{g} (k-2)}^{2} \le 2 \abs{\hat{f}(k-1)}^2 + 2 \abs{\hat{g} (k-1)}^{2} \leqslant  2 \abs{\hat{f}(k-1)}^2 + 2 D_{\sigma, k-1} (g).
    \end{equation*}
    Thus,
    \begin{equation*}
        D_{\zeta, k} (zf) \leqslant  2 \abs{\hat{f} (k-1)}^2 + (k+2) D_{\sigma, k-1} (g) = 2 \abs{\hat{f}(k-1)}^2 + (k+2) D_{\zeta, k} (f).
    \end{equation*}
    Integrating both sides with respect to the measure $\mu_1$ and using equation \eqref{eqn:local-dirichlet-integral}, we obtain the first part of the Proposition. 

Now the last part, namely the inclusion $\mathscr D_{\mu,k}\subset \mathscr D_{\mu,k-1}$ for each $k\in \mathbb Z_{\geqslant 1},$ follows immediately from the Proposition \ref{prop:diff-formula}.
\end{proof}

Let $k\in \mathbb Z_{\geqslant 1}.$ From Proposition \ref{prop:sigma-vs-mu},  we find that $ \mathscr D_{\mu, k} \subset \mathscr D_{\sigma, k-1},$ for each $\mu\in M_+(\overline{\D}).$ Since $\mathscr D_{\sigma,k-1}\subset H^2,$ we may define the norm on $\mathscr D_{\mu, k}$ by 
\begin{equation*}
    \norm{f}_{\mathscr D_{\mu, k}} ^2 := \norm{f}_{H^2} ^2 + D_{\mu, k} (f), \qquad f \in \mathscr D_{\mu, k}.
\end{equation*}
It is straightforward to verify that $\mathscr D_{\mu,k}$ is a Hilbert space with respect to the above norm. Let $n\in\mathbb Z_{\geqslant 0}.$ Since for each $w\in \D,$ the map $f\mapsto f^{(n)}(w)$ is a bounded linear functional on $H^2,$ it follows that the map $f\mapsto f^{(n)}(w)$ is also a bounded linear functional on $\mathscr D_{\mu,k}.$ This gives, in particular,  $\mathscr D_{\mu,k}$ is a reproducing kernel Hilbert space.
    
\begin{proposition}
       Let $\mu\in M_+(\overline{\D}).$ For each $k \ge 1$, the operator $M_z$ on $\mathscr D_{\mu, k} $ is a bounded and cyclic operator with cyclic vector $1.$ 
    \end{proposition}
    \begin{proof}
        Since $\mathscr D _{\mu, k}$ is a reproducing kernel Hilbert space, an application of the Proposition \ref{prop:boundedness-of-mu} together with closed graph theorem  will give us that $M_z$ on $\mathscr D_{\mu, k} $ is a bounded.
The proof of cyclicity of $M_z$ with cyclic vector $1$ follows immediately from Proposition \ref{prop:dilations} and the arguments used in Corollary \ref{cor:density-of-poly}.
    \end{proof}

We now study some properties of the multiplication operator $M_z$ on $\mathscr D_{\mu,k}.$ Recall that for a non negative integer $n\in \mathbb Z_{\geqslant 0}$ and a bounded operator $T$ acting on a Hilbert space $\mathcal H,$  the $n$-th backward difference $B_n(T)$ and the $n$-th forward difference $\beta_n(T)$ of is defined by
\begin{align*}
    B_n(T)= \sum\limits_{j=0}^n (-1)^j\binom{n}{j}{T^*}^j T^j,\quad  \beta_n(T)= \sum\limits_{j=0}^n (-1)^{n-j}\binom{n}{j}{T^*}^j T^j.
\end{align*} 

 An operator $T \in \mathcal{B}(\mathcal{H})$  is called completely hyperexpansive if $B_n(T) \leqslant 0$ for every $n \geqslant 1$,  see \cite{Atha95,1993multiplication,Athvale2000,Jablonski2002} and the references therein for further developments on completely hyperexpansive operators. More generally, an operator \(T\) is called a completely hyperexpansion of order \(k\) if \(B_n(T)\leqslant 0\) for every \(n \geqslant k\), see \cite[Sec. 3]{CS17}. In a similar manner, an operator \(T \in \mathcal{B}(\mathcal{H})\) is called completely hypercontraction if \(B_n(T) \geqslant 0\) for every \(n \geqslant 1\), see \cite{Aglerhyperc}. More generally, \(T\) is called a completely hypercontraction of order \(k\) if \(B_n(T) \geqslant 0\) for every \(n \geqslant k\); see \cite[Sec. 4]{CMF}. 

 We remark that every completely hyperexpansive operator of order $k$ is necessarily completely hyperexpansive of order $(k-1),$ whenever $k$ is an even positive integer, see  \cite[Lemma 3.13]{CS17} and \cite[Theorem 2.5]{2015Gu}. Similarly every completely hypercontractive operator of order $k$ is necessarily completely hypercontractive of order $(k-1),$ whenever $k$ is an odd positive integer, see \cite[Theorem 2.5]{2015Gu}. 

In order to unify and study these two classes of operators together, namely complete hyperexpansion and complete hypercontraction of finite order, and motivated by the terminology in \cite[IV, Definition 1.2]{1993multiplication},  we introduce the notion of \textit{Dirichlet type operators of finite order}.  Let $k\in \mathbb Z_{\geqslant 1}.$ An operator $T \in \calB \left( \calH \right)$ is called a Dirichlet type operator of order $k$ if $(-1)^kB_{n} \left( T \right) \ge 0$ for all $n \ge k.$  In the following Proposition we find that the operator $M_z$ on $\mathscr D_{\mu, k}$ is a Dirichlet type operator of order $k.$ 
\begin{proposition}\label{thm:shift-completely-hyperstuff}
 Let $k\in \mathbb Z_{\geqslant 1}$  and $\mu\in M_+(\overline{\D}).$ The operator $M_{z}$ on $\mathscr D_{\mu, k}$ is a Dirichlet type operator of order $k.$ Moreover $M_{z}$ is a $(k+1)$-isometry if and only if $\mu$ is supported on $\mathbb T.$
	    \end{proposition}
    \begin{proof}
     Let $\mu=\mu_1+\mu_2$ where $\mu_1 = \mu |_{\D}$ and $\mu_2 = \mu |_{\T}$. Suppose $f\in \mathscr D_{\mu, k}.$  Proposition \ref{prop:diff-formula} along with an induction argument will give us 
\begin{align}\label{Less difference}
  \sum_{j=0}^{n} \left( -1 \right)^{j} \binom{n}{j} D_{\mu,k}(z^{j}f) 
&= \left( -1 \right)^{n} D_{\mu, k-n} \left( f \right),\quad 1\leqslant n\leqslant k.  
\end{align}
From \cite[Eq. (4.1), p.461]{GGR22}, we also have
 \begin{align*}
     \sum_{j=0}^{n} \left( -1 \right)^{j} \binom{n}{j} D_{\mu_2,k}(z^{j}f)= 0,\quad \mbox{for every\,}\, n> k.
 \end{align*}
 Therefore, again an induction argument together with Proposition \ref{prop:diff-formula}, will give 
   \begin{align}\label{more difference}
  \sum_{j=0}^{n} \left( -1 \right)^{j} \binom{n}{j}  D_{\mu,k}(z^{j}f)  &= \left( -1 \right)^{k} \int_{\D} \left( 1-\abs{z}^{2} \right)^{n-k} \abs{f\left( z \right)}^{2} \mathrm{d}\mu_1 \left( z \right),\quad n > k. 
\end{align}  
The norm of $f\in\mathscr D_{\mu,k}$ is given by
\begin{align*}
    \|f\|^2= \|f\|_{H^2}^2 + D_{\mu,k}(f).
\end{align*}
 Note that $\|z^jf\|_{H^2}^2=\|f\|_{H^2}^2$ for every $j\in \mathbb Z_{\geqslant 1}.$ Now it is straightforward to verify that 
     \begin{align*}
        \langle B_n(M_z)f,f\rangle = \begin{cases}
            (-1)^kD_{\mu,0}(f), & \mbox{if}\,\,n=k,\\
            \left( -1 \right)^{k} \displaystyle \int_{\D} \left( 1-\abs{z}^{2} \right)^{n-k} \abs{f\left( z \right)}^{2} \mathrm{d}\mu_1 \left( z \right), &  \mbox{if}\,\, n > k.
        \end{cases}
     \end{align*}
     Thus if $k$ is an odd positive integer, $B_n(M_z)\leqslant 0$ for every $n\geqslant k.$ And if $k$ is an even positive integer, then $B_n(M_z)\geqslant 0$ for every $n\geqslant k.$ Moreover it follows that $B_{k+1}(M_z)=0$ if and only if $\mu_1=0,$ that is, $\mu$ is supported on $\mathbb T.$
    \end{proof}
For various spectral properties of a Dirichlet type operator of finite order, we refer the readers to \cite[Sec. 4]{CMF} and \cite[Sec. 3]{CS17}.  The study of Dirichlet type operator of order $m$ is intimately related to the completely alternating functions of finite order and completely monotone functions of finite order,  see \cite[Sec. 2]{CS17}, \cite[Sec. 3]{CMF}. A sequence $\{\psi(n)\}_{n \in \mathbb Z_{\geqslant 0}}$ of real number is called completely alternating of order $m$ if $(\triangledown^j \psi)(n) \leqslant 0$ for every $j\geqslant m$ and $n\in \mathbb Z_{\geqslant 0},$ where the backward difference sequence $(\triangledown \psi)$ is defined by $(\triangledown \psi)(n)= \psi(n)-\psi(n+1)$ for $n\in \mathbb Z_{\geqslant 0}$ and $(\triangledown ^j\psi)$ are inductively defined by 
\begin{align*}
    (\triangledown ^0 \psi)=\psi, \qquad (\triangledown ^{j+1}\psi)= \triangledown (\triangledown^j \psi),\qquad j\in \mathbb Z_{\geqslant 0}.
\end{align*}
On the other hand, a sequence $\{\psi(n)\}_{n \in \mathbb Z_{\geqslant 0}}$ of real number is called completely monotone of order $m$ if $(\triangledown^j \psi)(n) \geqslant 0$ for every $j\geqslant m$ and $n\in \mathbb Z_{\geqslant 0}.$ It is straightforward to verify that an operator $T\in \mathcal B(\calH)$ is a Dirichlet type operator of order $(2k-1)$ if and only if the sequence $\psi(n)= \|T^nh\|^2$ for $n\in \mathbb Z_{\geqslant 0},$ is completely alternating of order $(2k-1)$ for every $h\in \mathcal H.$  Similarly, an operator $T\in \mathcal B(\calH)$ is a Dirichlet type operator of order $2k$ if and only if the sequence $\psi(n)= \|T^nh\|^2$ for $n\in \mathbb Z_{\geqslant 0},$ is completely monotone of order $2k$ for every $h\in \mathcal H.$ In the following lemma we find a necessary condition on the growth of the sequence $\psi(n)=\|T^nh\|^2$ when $T$ is a Dirichlet type operator of order $m.$
\begin{lemma}\label{order of CAk}
    Let $m\in\mathbb Z_{\geqslant 1}$ and $T\in \mathcal B(\calH)$ be a Dirichlet type operator of order $m.$ Then for each $h\in\mathcal H,$ we have that $\|T^nh\|^2=O(n^{m}).$
\end{lemma}
\begin{proof}
 Let $\psi(n)= \|T^nh\|^2$ for $n\in \mathbb Z_{\geqslant 0}.$ First we consider the case when $m$ is an odd positive integer. In this case $\psi$ is a completely alternating sequence of order $m,$ from \cite[Theorem 2.5]{CS17}, we find that there exists a polynomial $p_m$ of degree at most $m$ and a finite positive Radon measure $\mu$ on $[0,1]$ such that for every $n\in \mathbb Z_{\geqslant 0},$
 \begin{align*}
   \psi(n) = p_{m}(n) + \int\limits_{[0,1]}\sum_{j=0}^{n-m-1}\left((-1)^{m}\frac{(x-1)^{j}}{(j+m+1)!}(n)_{\overline{j+m+1}}\right)d\mu(x),
\end{align*}
    where the integral in the last expression is absent if $n < m+1$ and $(a)_{\overline{j}}=a(a-1)\cdots(a-j+1)$ denotes the falling factorial for any $a\in \mathbb R$. We also use $(a)^{\overline{j}}=a(a+1)\cdots(a+j-1)$ to denote the rising factorial for any $a\in \mathbb R.$ We claim that for every $n\in  \mathbb Z_{\geqslant 0},$
    \begin{equation}\label{CA representation}
	\psi \left( n \right) = p_{m} \left( n \right) + \left( -1 \right)^{m} \sum_{j=0}^{n-m-1} \frac{\left( n-m-j \right)^{\overline{m}}}{m!} \hat{\mu} \left( j \right)
    \end{equation}
    where 
    $\hat{\mu} \left( j \right) = \int\limits_{[0,1]} x^{j} \mathrm{d}\mu \left( x \right)$.
    Observe that to prove the claim it suffices to show that for each $n \in \mathbb{Z}_{\ge 0}$ and $x \in \mathbb R$,
    \begin{equation}\label{magic identity}
	\sum_{j=0}^{n-m-1} \binom{n}{j+m+1} \left( x-1 \right)^{j} = 
	\sum_{j=0}^{n-m-1} \binom{n-j-1}{m} x^{j}.
    \end{equation}
The above identity can easily be verified by an induction argument and we omit the details.  Since $m$ is an odd positive integer and $\mu$ is a finite positive measure on $[0,1],$ using \eqref{CA representation}, we obtain that for each $n \in \mathbb Z_{\geqslant 0},$ 
    \begin{equation*}
	0 \leqslant \psi \left( n \right) \le p_{m} \left( n \right) = O \left( n^{m} \right).
    \end{equation*}  This completes the proof for the case when $m$ is an odd positive integer. Now consider that $m$ is an even positive integer. Using the characterization of the completely monotone sequences of finite order as given in \cite[Theorem 3.6]{CMF}, we get that there exists a polynomial $p_{m-1}$ of degree at most $(m-1)$ and a finite positive Radon measure $\mu$ on $[0,1]$ such that for every $n\in \mathbb Z_{\geqslant 0},$
    \begin{align*}
        \psi(n) &= p_{m-1}(n) + \int\limits_{[0,1]} \sum_{j=0}^{n-m} \frac{(x - 1)^j}{(j + m)!}  (n)_{\overline{j+m}} \, d\mu(x),\\
        &= p_{m-1} \left( n \right) +  \sum_{j=0}^{n-m} \frac{\left( n-m-j-1 \right)^{ \overline{m-1} }}{\left( m-1 \right)!} \hat{\mu} \left( j \right)
    \end{align*}
Here the second equality follows from the identity \eqref{magic identity}. Hence we obtain that 
\begin{equation*}
	\abs{\psi \left( n \right)} \le |p_{m-1} \left( n \right)|+ \frac{\mu \left( [0,1] \right)}{(m-1)!} \left( n-m-1 \right)^{ \overline{m-1}} \left( n-m+1 \right). 
    \end{equation*}
As $p_{m-1}$ is a polynomial of degree at most $(m-1),$ it follows that $\psi(n)=O(n^m).$ This completes the proof of the lemma.
\end{proof}

\begin{remark}\label{another example}
    To provide another class of examples of operator of Dirichlet type of finite order, we consider the operator $M_z$ on $\mathscr D_{\alpha}$ as introduced in the Introduction. Let $k\in\mathbb Z_{\geqslant 1}$ and $\alpha \in [k-1, k].$ Then $M_z$ is a Dirichlet type operator of order $k.$ This follows by considering the function $f: \R \to \R$ given by $f(x)= (x+1)^{\alpha}$ and using the higher order mean value theorem which says if $f : \R_{\geqslant 0} \to \R$ is a $C^{\infty}$ function then for each $n \ge 1$, $h > 0$, $x \in \R$, there is some $y \in \left( x, x+nh \right)$ such that 
	    $\frac{1}{h^{n}} \Delta ^{n} _{h} f \left( x \right) = f^{\left( n \right)} \left( y \right),$ where $\Delta ^1 _{h} f (x) := f(x+h) - f(h)$ and for $n \ge 2$, $\Delta ^n _h f(x)$ is defined recursively by $\Delta ^n _h f(x): =\Delta ^1 _h \Delta ^{n-1} _ h f(x)$.
\end{remark}

\section{Representation of Cyclic Dirichlet Type Operators of Finite Order}

Let $\mathcal D$ denote the Fr\'{e}chet space of $C^{\infty}$ functions on $\T,$ and let $\mathcal D'$ denote its topological dual, i.e., the space of distributions on the unit circle $\mathbb T.$  For $k\in\mathbb Z,$ let $\hat{\mu}(k)= \mu (\Bar{\zeta}^k)$ be the standard Fourier coefficients of $\mu\in \mathcal D'.$ The Poisson extension of $\mu\in \mathcal D'$ is the harmonic function $P_{\mu}(z)$ defined by
 \begin{align*}
    P_{\mu}(z):= \sum\limits_{k=0}^{\infty} \hat{\mu}(k)z^k + \sum\limits_{k=1}^{\infty} \hat{\mu}(-k)\Bar{z}^k,\quad z\in\mathbb D.
 \end{align*}
 Given $\mu\in \mathcal D'$ and $f\in \mathbb C[z],$ the set of analytic polynomials, the weighted Dirichlet integral of order $k\in \mathbb Z_{\geqslant 0}$ is defined by
 \begin{align*}
    D_{\mu,k}(f)= 
    \begin{dcases}
         \lim_{r\to 1^{-}}\frac{1}{k! (k-1)!} \int_{r\D}  \abs{f^{\left( k \right)}\left( z\right) }^2 P_{\mu}(z) (1-\abs{z}^2)^{k-1} \, \mathrm{d}A(z), & \mbox{if\,\,\,} k\in\mathbb Z_{\geqslant 1},\\
         \lim_{r\to 1^{-}} \int_{\T}  \abs{f\left( r\zeta\right) }^2 P_{\mu}(r\zeta) \, \mathrm{d}\sigma(\zeta), & \mbox{if\,\,\,} k=0.
     \end{dcases}
 \end{align*}
The following formula provides a convenient way to compute $D_{\mu,k}(f)$ in terms of the Fourier coefficients of both $\mu$ and $f,$ see \cite[Lemma 3.2]{Rydhe19} for the details:
\begin{align}\label{Formula for semi norm}
   D_{\mu,k}(f)= \sum\limits_{j,l=k}^{d}\binom{j\wedge l}{k} \hat{f}(j) \overline{\hat{f}(l)} \hat{\mu}(l-j), \qquad k \in \mathbb Z_{\geqslant 0},
\end{align}
where $j\wedge l= \min\{j,l\}$ and $f \in \mathbb C[z]$ is expressed as $f(z)= \sum_{j=0}^{d}\hat{f}(j) z^j.$  
We refer the reader to \cite{Rydhe19} for additional properties of the higher-order weighted Dirichlet integral.

Let $\overrightarrow{\mu}= (\mu_0,\mu_1,\ldots,\mu_{m})\in (\mathcal D')^{m}\times M_+(\overline{\D})$ and consider the sesquilinear form 
\begin{align*}
    \|f\|^2_{_{\overrightarrow{\mu}}}:= \sum\limits_{j=0}^{m} D_{\mu_j,j}(f), \quad f\in \mathbb C[z].
\end{align*}
If there exists a $C>0$ such that 
\begin{align}\label{allowable}
 0 \leqslant  \|zf\|^2_{_{\overrightarrow{\mu}}}  \leqslant C  \|f\|^2_{_{\overrightarrow{\mu}}},\quad f\in \mathbb C[z],
\end{align}
then we say  $\overrightarrow{\mu}= (\mu_0,\mu_1,\ldots,\mu_{m})$ is an \textit{allowable $(m+1)$-tuple} in $(\mathcal D')^{m}\times M_+(\overline{\D})$. Moreover we say $\overrightarrow{\mu}$ is \textit{normalized} if $\|1\|^2_{_{\overrightarrow{\mu}}}= \hat{\mu_0}(0)=1.$ For any such  allowable tuple, let $\mathcal K_{\overrightarrow{\mu}}= \{f\in \mathbb C[z]: \|f\|_{_{\overrightarrow{\mu}}} =0\}$ and define $\mathscr D_{\overrightarrow{\mu}}$ as the completion of $\mathbb C[z]/\mathcal K_{\overrightarrow{\mu}}$ with respect to the norm induced by the sesquilinear form $\|\cdot\|{_{\overrightarrow{\mu}}}.$ Note that the operator $M_z$ is well defined on  $\mathbb C[z]/\mathcal K_{\overrightarrow{\mu}}$ and extends uniquely as a bounded linear operator on  $\mathscr D_{\overrightarrow{\mu}}.$ Now we are ready to state the other main theorem of this article, namely the model theorem for cyclic  Dirichlet operator type of order $m.$  

In what follows, we show that $M_z$ is a cyclic Dirichlet-type operator of order $m,$ and we will establish that $(M_z,\mathscr D_{\overrightarrow{\mu}})$ serves as a model for every cyclic Dirichlet-type operator of order $m.$

\begin{proof}[Proof of Theorem \ref{Model Theorem}:]
    Let $\overrightarrow{\mu}= (\mu_0,\mu_1,\ldots,\mu_{m})\in (\mathcal D')^{m}\times M_+(\overline{\D})$ be a normalized allowable $(m+1)$-tuple. Further write $\mu_m=\mu_{_{m,1}}+\mu_{_{m,2}}$ where $\mu_{_{m,1}} = {\mu_m} |_{\D}$ and $\mu_{_{m,2}} = {\mu_m} |_{\T}$. If $\mu\in \mathcal D',$ then from \cite[Proposition 3.4]{Rydhe19}, it follows that for every $f\in\mathbb C[z],$
    \begin{align}\label{difference higher order}
        \sum_{j=0}^{n} \left( -1 \right)^{j} \binom{n}{j} D_{\mu_{_{m,2}},k}(z^{j}f)= 0,\quad n> m.
    \end{align}
Using \eqref{Less difference} and \eqref{more difference}, we then have for every $f\in\mathbb C[z],$
\begin{align*}
        \sum\limits_{r=0}^{m}\sum_{j=0}^{n} \left( -1 \right)^{j} \binom{n}{j} D_{\mu_r,r}(z^{j}f) &= 
        \begin{cases}
         \left( -1 \right)^{m}  D_{\mu_m,0}(f) , & \quad \mbox{if\,\,} n=m,\\
        \left( -1 \right)^{m} \displaystyle \int_{\D} \left( 1-\abs{z}^{2} \right)^{n-m} \abs{f\left( z \right)}^{2} \mathrm{d}{\mu_{_{m,1}}} \left( z \right), &\quad \mbox{if\,\,} n > m.
        \end{cases} 
    \end{align*}
    It follows that if $m$ is an odd positive integer, $B_n(M_z)\leqslant 0$ for every $n\geqslant m.$ And if $m$ is an even positive integer, then $B_n(M_z)\geqslant 0$ for every $n\geqslant m.$ Therefore $M_z$ is a cyclic Dirichlet-type operator of order $m.$ Moreover it is straightforward to see that $B_{m+1}(M_z)=0$ if and only if $\mu_{_{m,1}}=0,$ that is, $\mu_m$ is supported on $\mathbb T.$

For the converse direction, let $T\in\mathcal B(\calH)$ be an operator of Dirichlet type of order $m$ with a cyclic unit vector $e.$ The following lemma gives the existence of a normalized  allowable $(m+1)$-tuple $\overrightarrow{\mu} = \left( \mu_{0}, \mu_{1}, \ldots , \mu_{m} \right).$ 
\begin{lemma}\label{existence of distributions}
    Let $m\in \mathbb Z_{\geqslant 1}$ and $T\in \mathcal B(\calH)$ be an operator of Dirichlet type of order $m$ with unit cyclic vector $e.$ Then the following holds:
    \begin{itemize}
    \item[\rm (i)] There exists a positive measure  $\mu_{m}\in M_+(\overline{\D})$ such that 
 \begin{equation*}
	\ip{\beta_{m} (T) T^k e, T^l e}_{\calH} = \int_{\mathbb{\overline{\D}}} z^{k} \overline{z}^{l} \mathrm{d}\mu_{m} (z) \qquad k,l \in \Z _{\ge 0}.
    \end{equation*}
    \item[\rm (ii)] There exist distributions $\mu_j\in \mathcal D'$  for $j=0,1,\ldots,(m-1)$ such that
    \begin{equation*}
    \hat{\mu}_{j} \left( k \right) = \overline{\hat{\mu}_{j} \left( -k \right)} = \ip{\beta_{j} \left( T \right) e, T^{k}e}_{\calH} \qquad  k \in \mathbb Z_{\geqslant 0}.
\end{equation*}
\end{itemize}
\end{lemma}
\begin{proof}
   $(\rm i)$ For any polynomial $p\in\mathbb C[z]$ with $p(z)= \sum_{j=0}^lc_jz^j,$ we  define $p(T)e:=\sum_{j=0}^lc_jT^j e.$ Since $T$ is a Dirichlet type operator of order $m$, we have $\beta_{m}(T)\geqslant 0.$  Consider the semi-inner product on $\mathbb C[z]$ given by
   \begin{align*}
       \ip{p,q}_{\beta_m(T)}:= \ip{\beta_m(T) p(T)e, q(T)e}_{\calH},\qquad p,q\in \mathbb C[z].
   \end{align*}
   Let $\mathcal W_T$ be the Hilbert space obtained as the completion of the quotient space $\mathbb C[z]/{\mathcal N_T}$ with respect to the above semi-inner product, where $\mathcal N_T=\ker (\beta_m(T)).$ As $T^*\beta_{k}(T)T-\beta_k(T)=\beta_{k+1}(T),$ by an induction argument it follows that 
   \begin{align*}
       \ip{\beta_n(M_z) p,p}_{\beta_m(T)}= \ip{\beta_{m+n}(T)p(T)e, p(T)e}_{\calH},\qquad p\in\mathbb C[z],\,n\geqslant 0.
   \end{align*}
   Since $T$ is of Dirichlet type operator of order $m,$ it follows that $B_{n}(M_z)\geqslant 0$ for every $n\in\mathbb Z_{\geqslant 0}.$ It follows from \cite[Theorem 3.1]{Aglerhyperc}, $M_z$ extends to a unique contractive subnormal operator on $\mathcal W_T,$ which we will also denote it by $M_z.$ Moreover, the equivalence class $[1]$ in  $\mathbb C[z]/{\mathcal N_T},$ determined by the the constant polynomial $1,$ is a cyclic vector for $M_z$  on $\mathcal W_T.$ Thus $M_z$ can be extended to a normal operator $N$ on a Hilbert space $\mathcal K_T$ which contains $\mathcal W_T$ as a closed subspace and $N$ is the minimal normal extension of $M_z$ so that $\|N\|=\|M_z\|\leqslant 1.$ Let $E$ be the spectral measure of $N$ and $\mu_m$ be the scalar spectral measure given by $\mu_m(A)= \ip{E(A)[1],[1]}_{\mathcal K_T}$ for every Borel set $A$ in the complex plane. Since $\|N\|\leqslant 1,$ the measure $\mu_m$ is supported on $\overline{\D}.$ It also follows that for every $p,q\in\mathbb C[z],$
   \begin{align*}
       \int_{\overline{\D}}p(z)\overline{q(z)} d\mu_m(z)&= \ip{p(N)[1],q(N)[1]}_{\mathcal K_T}\\
       &= \ip{p,q}_{\beta_m(T)}\\
     &=\ip{\beta_m(T)p(T)e,q(T)e}_{\mathcal H}.
   \end{align*}
In particular taking $p(z)= z^k$ and $q(z)=z^l,$ we obtain the required identity. 

(\rm ii) Consider the sequence $\{x_n\}_{n\geqslant 0}$ given by $x_n:=\norm{T^ne}^2$  for $n\in\mathbb Z_{\geqslant 0}.$ If $T$ is an operator of Dirichlet type of order $m$ for some $m\in\mathbb Z_{\geqslant 1},$ from Lemma \ref{order of CAk} it follows that  $x_n=O(n^{m}).$  Therefore we have that 
\begin{align*}
    |\ip{\beta_j(T)e, T^ne}|\leqslant \|\beta_j(T)\|\|T^ne\| \lesssim (1+n)^{\frac{m}{2}}, \qquad n\in \mathbb Z_{\geqslant 0}.
\end{align*} 
for each $j=0,1,\ldots,m-1.$ Hence there exist real distributions $\mu_j\in\mathcal D'$ satisfying 
\begin{equation*}
    \hat{\mu}_{j} \left( n \right) = \overline{\hat{\mu}_{j} \left( -n \right)} = \ip{\beta_{j} \left( T \right) e, T^{n}e}_{\calH} \qquad  n \in \mathbb Z_{\geqslant 0},
\end{equation*} for each $j=0,1,\ldots,m-1.$  This completes the proof of the Lemma.
\end{proof}

We consider the tuple $\overrightarrow{\mu} = \left( \mu_{0}, \mu_{1}, \ldots , \mu_{m} \right) \in (\mathcal D')^{m}\times M_+(\overline{\D})$ as obtained from Lemma \ref{existence of distributions}. We will show that
    \begin{equation}\label{norm formula}
	\norm{p\left( T \right) e}_{\calH}^{2}= \sum_{j=0}^{m} D_{\mu_{j} , j} \left( p \right),\qquad p\in \mathbb C[z],
    \end{equation}
where $e$ is the unit cyclic vector for $T.$ Let $\mathcal L$ be the subspace of $\calH$ given by $\mathcal L= \{p(T)e: p\in \mathbb C[z]\}.$ Let $\hat{L}$ be the backward shift operator on $\mathcal L$ defined by 
\begin{align*}
    \hat{L}\left (\sum\limits_{j=0}^nc_jT^je \right )= \sum\limits_{j=1}^nc_jT^{j-1}e,\qquad c_j\in\mathbb C.
\end{align*}
For any linear map $A$ on $ \mathcal L,$ we define the sesquilinear form $\ip{S(A)x,y}$ on $\mathcal L$ by
\begin{align*}
  \ip{S(A)x,y}:= \sum\limits_{j=1}^{\infty} \ip{A \hat{L}^jx,\hat{L}^jy},\qquad x,y\in \mathcal L.
\end{align*}
In order to establish \eqref{norm formula}, we need the following identities. 
\begin{lemma}\label{moment through diff}
Let $p\in\mathbb C[z].$ Then for any $r=0,1,\ldots,(m-1),$ we have
\begin{align*}
   D_{{\mu_r},0}(p) = \begin{cases}
       \ip{\Big(\beta_r(T)-S(\beta_{r+1}(T))\Big)p(T)e,p(T)e}, & \mbox{if\,\,}\,\, 0\leqslant r \leqslant m-1,\\
       \ip{\beta_m(T) p(T)e, p(T)e}, & \mbox{if\,\,}\,\, r =m.
   \end{cases}
\end{align*}
\end{lemma}
\begin{proof}
Note that $\hat{L}e=0$ and $(\hat{L}T)(q)=q$ for every $q\in \mathbb C[z].$ Since $T^*\beta_{r}(T)T-\beta_r(T)=\beta_{r+1}(T),$  we find that  
\begin{align*}
        \ip{\Big (S(\beta_{r+1}(T))\Big) T^ke,T^le}  &=   \sum\limits_{j=1}^{k\wedge l} \ip{\beta_{r+1}(T) T^{k-j}e,T^{l-j}e}\\
        &= \sum\limits_{j=1}^{k\wedge l} \ip{\beta_{r}(T) T^{k-j+1}e,T^{l-j+1}e} - \sum\limits_{j=1}^{k\wedge l} \ip{\beta_{r}(T) T^{k-j}e,T^{l-j}e}\\
        &= \ip{\beta_r(T)T^{k}e,T^le}- \ip{\beta_r(T) T^{k- (k\wedge l)}, T^{l- (k\wedge l)}},\\
        &=\ip{\beta_r(T)T^{k}e,T^le}- \hat{\mu}_r(l-k),
    \end{align*}
    for each $r=0,1,\ldots,(m-1)$and $l,k\in\mathbb Z_{\geqslant 0.}$ Note that the last equality follows from Lemma \ref{existence of distributions} (ii). In view of the formula of $D_{\mu_r,0}(p)$ mentioned in \eqref{Formula for semi norm} we thus obtain that 
    \begin{align*}
         D_{\mu_r,0}(p)= \ip{\Big (\beta_r(T)-S(\beta_{r+1}(T))\Big) p(T)e, p(T)e},\qquad r=0,1,\ldots,(m-1).
    \end{align*} 
    From Lemma \ref{existence of distributions} (i), it also follows that 
    $D_{\mu_m,0}(p)= \ip{\beta_{m}(T) p(T)e, p(T)e}.$ This completes the proof of the lemma.
    \end{proof}
From \cite[Proposition 3.4]{Rydhe19}, we have that for any $\mu\in \mathcal D',p\in \mathbb C[z]$ and $k\in \mathbb Z_{\geqslant 0},$
\begin{align*}
    D_{\mu,k}(zp)-D_{\mu,k}(p) = \begin{cases}
        0, & \mbox{if}\,\,k=0,\\
        D_{\mu,k-1}(p), & \mbox{if}\,\,k\geqslant 1.
    \end{cases}
\end{align*}
In view of the above identity and using the fact $L^n(p)=0$ for $n>$ deg$(p),$ we obtain that for  any $\mu\in \mathcal D',p\in \mathbb C[z]$ and $k\in \mathbb Z_{\geqslant 0},$
 \begin{equation}\label{one step up for dist}
	\sum_{j=1}^{\infty} D_{\mu, k} \left( L^{j}p \right) = D_{\mu, k+1} (p),
    \end{equation}
see the arguments of Proposition \ref{one step up formula for measure} for a proof.  Using \eqref{one step up for dist} along with an induction argument, we find that for $\mu_r\in \mathcal D'$ as obtained in Lemma \ref{existence of distributions} (ii),
 \begin{align*}
    \ip{S^j\Big(\beta_r(T)- S(\beta_{r+1}(T))\Big) p(T)e,p(T)e}= D_{\mu_r,j}(p),\qquad p\in\mathbb C[z],\,j\in \mathbb Z_{\geqslant 0}.
 \end{align*} 
 This gives us that for any $p\in\mathbb C[z],$
 \begin{align*}
  \sum\limits_{r=0}^{m} D_{\mu_r,r}(p) &= \Bigg (\sum\limits_{r=0}^{m-1} \ip{S^r\Big(\beta_r(T)- S(\beta_{r+1}(T))\Big) p(T)e,p(T)e}\Bigg) + \ip{\beta_{m}(T) p(T)e, p(T)e} \\
  &= \norm{p(T)e}^2. 
 \end{align*}
  Since $\mathcal L$ is dense in $\mathcal H$ and $T$ is a bounded operator, it follows that   $\overrightarrow{\mu} = \left( \mu_{0}, \mu_{1}, \ldots , \mu_{m} \right)$ is an allowable $(m+1)$-tuple. Moreover this is a normalized allowable $(m+1)$-tuple as $e$ is a unit vector. The identity \eqref{norm formula} also shows that the map $U: \mathcal L \to \mathbb C[z]$ given by 
  \begin{align*}
     U(p(T)e):= p,\qquad p\in\mathbb C[z],
  \end{align*} defines an isometry from $\mathcal L$ into $\mathscr D_{\overrightarrow{\mu}}.$ Since $\mathcal L$ is dense in $\mathcal H$ and $\mathbb C[z]$ is dense in $\mathscr D_{\overrightarrow{\mu}},$ the map $U$ extends to an unitary map from $\mathcal H$ onto $\mathscr D_{\overrightarrow{\mu}}.$ Moreover it is straightforward to see that $M_zU=UT.$

For the proof of the last part of the Theorem \ref{Model Theorem}, let $T_j,$  be two cyclic Dirichlet type operator of order $m$ with unit cyclic vector $e_j$ for $j\in\{1,2\}.$ Let  $\overrightarrow{\mu}_j$ be the associated normalized allowable $(m+1)$-tuple for $j=1,2,$ as obtained in the first part of the Theorem. Suppose $V:\mathcal H_1\to \mathcal H_2$ be a unitary such that $VT_1=T_2V$ and $Ve_1=e_2.$ Then in view of the Lemma \ref{moment through diff}, it follows that $\overrightarrow{\mu}_1= \overrightarrow{\mu}_2.$ Conversely, if $\overrightarrow{\mu}_1= \overrightarrow{\mu}_2,$ from \eqref{norm formula}, it follows that $U:\mathcal H_1 \to \mathcal H_2$ given by $U(p(T_1)e)=p(T_2)e_2,$ for $p\in\mathbb C[z]$ defines a unitary from $\calH_1$ onto $\calH_2$ satisfying $UT_1=T_2U$ and $Ue_1=e_2.$
\end{proof}

\begin{remark}\label{positivity of measure}
Let $T\in \mathcal B(\mathcal H)$ be a cyclic Dirichlet type operator of order $m$ with cyclic unit vector $e$ and $T$ be unitarily equivalent to $M_z$ on $\mathscr D_{\overrightarrow{\mu}}$ where $\overrightarrow{\mu}=(\mu_0,\mu_1,\ldots,\mu_m)$ is the associated allowable normalized $(m+1)$-tuple as given in Theorem \ref{Model Theorem}. Then from Lemma \ref{moment through diff}, it follows that for each $r=0,1,\ldots,(m-1)$ the distribution $\mu_r$ is a non negative distribution on $\T$ (or equivalently $\mu_r$ is a non negative measure in $M_+(\mathbb T)$) if and only if $\ip{S(\beta_{r+1}(T))x,x}_{\calH}\leqslant \ip{\beta_r(T)x,x}_{\calH}$ for every $x\in \mathcal H.$ It also follows from equation \eqref{difference higher order} and Proposition \ref{thm:shift-completely-hyperstuff} that $\mu_m$ is supported on $\mathbb T$ if and only if $\beta_{m+1}(T)=0,$ that is, $T$ is a cyclic $(m+1)$-isometry.
\end{remark}

\begin{example}
    Let $\calH = \mathscr  D _{\alpha}$ as defined in equation \eqref{eq:D-alpha}, $e=1 \in \mathscr D_{\alpha}$ and $T=M_z \in \mathcal B(\mathscr D_{\alpha})$. By Remark \ref{another example}, it follows that $T$ is a Dirichlet type operator of order $m$ where $m$ is the unique integer satisfying $m-1 < \alpha \le m$. Here we describe the distributions of the circle $\mu_j$, $0\le j \le m-1$ and the measure on the closed unit disc $\mu_m$ as in Theorem \ref{Model Theorem}  for this operator $T$. 
    
    Set $x_n := (n+1)^{\alpha}$ and take $m \in \Z$ such that $m-1 < \alpha \le m$. Recall $\sigma$ denotes the normalized Lebesgue measure on $\T.$ We claim that 
\begin{equation*}
    \mu_j := \begin{cases}
        \left( \Delta ^j x_0 \right) \sigma, & 0\le j \le m-1, \\
        (m!) \sigma, & j=m, \, \alpha=m, \\
        \frac{(-1)^m}{\Gamma(-\alpha)} (1-\abs{z}^2)^{m} (-\log |z|^2)^{-\alpha-1} \, dA(z), & j=m,\,  m-1 < \alpha < m. 
    \end{cases}
\end{equation*}
It can be easily seen that 
\begin{equation*}
    \beta_j(T)z^k
    =\sum_{l=0}^j(-1)^{j-l}\binom{j}{l}T^{*l}T^l z^k
    =\sum_{l=0}^j(-1)^{j-l}\binom{j}{l}\frac{x_{k+l}}{x_k}z^k
    =\frac{\Delta^j x_k}{x_k}z^k.
\end{equation*}

We proceed to compute $\mu_j$ when $0\le j \le m-1$. By Lemma \ref{existence of distributions}, we have that
\begin{equation*}
    \hat{\mu}_{j} \left( k \right) = \overline{\hat{\mu}_{j} \left( -k \right)} = (\Delta^j x_0) \delta_{0, k}, \qquad k \in \Z_{\ge 0}.
\end{equation*}

By uniqueness of the Fourier coefficients of a distribution $\mathbb T$, we have that $ \mu_j = \left( \Delta ^j x_0 \right) \sigma$ for each $j=0,1,\ldots,m-1.$

We now consider the case $j=m$ . By \ref{existence of distributions}, we have 
\begin{equation*}
    \int_{\overline\D} z^k\overline z^{\ell}\,d\mu_m(z) =\ip{\beta_m(T)z^k,{z^{l}}}
=\frac{\Delta^m x_k}{x_k}\ip{z^k, z^{l}}
=\delta_{kl}\,\Delta^m x_k =  \delta_{kl}\,\Delta^m (k+1)^{\alpha}
\end{equation*}

Moreover, if $\alpha =m$, an inductive argument shows that $\Delta ^{m} (k+1)^m = m!$ and this shows that $\mu_m = (m!)\, \sigma$.

We now consider the case $m-1 < \alpha < m$. We use the following identity valid for every $C^m$ function $f$ on an interval containing $[x,x+m]$:
\begin{equation} 
\Delta^m f(x) =\int_{[0,1]^m} f^{(m)}(x+t_1+\cdots+t_m) \, dt_1 \cdots dt_m.
\label{eq:difference-to-integral}
\end{equation}
The proof of the identity is an application of Fundamental Theorem of Calculus and induction.

Applying \eqref{eq:difference-to-integral} to $f(x)=x^\alpha$ and $x=n+1$, we have
\begin{equation}
    \Delta^m (n+1)^\alpha
=\alpha(\alpha-1)\cdots(\alpha-m+1)
\int_{[0,1]^m}(n+1+t_1+\cdots+t_m)^{\alpha-m}\,dt_1 \, \ldots \, dt_m.
\label{eqn:delta-m-n+1-alpha}
\end{equation}

For a fixed positive number $s\in \R$, we have
\begin{equation*}
    \int_{0}^{\infty} e^{-su} u^{m-\alpha-1} \, du = \frac{\Gamma(m-\alpha)}{s^{m-\alpha}},
\end{equation*} 
where $\Gamma (z)$ denotes the Gamma function. Take $s:=n+1+t_1+\cdots + t_m$. Clearly $s > 0$ and
\begin{equation*}
    \left( n + 1 + t_1 + \cdots + t_m \right)^{\alpha-m} = \frac{1}{\Gamma(m-\alpha)} \int_{0}^{\infty} e^{-(n+1+t_1 + \cdots + t_m) u} u^{m-\alpha-1} du.
\end{equation*}

Substituting in the equation \eqref{eqn:delta-m-n+1-alpha}, we have
\begin{align*}
    \Delta^m (n+1)^{\alpha} &= \frac{\alpha(\alpha-1)\cdots(\alpha-m+1)}{\Gamma(m-\alpha)} \int\limits_{[0,1]^m} \int_{0}^{\infty} e^{-(n+1+t_1 + \cdots + t_m) u} u^{m-\alpha-1} \,du\, dt_1\, \ldots\, dt_m \\
    &= \frac{(-1)^m}{\Gamma(-\alpha)} \int_{0}^{\infty} \int\limits_{[0,1]^m} e^{-(n+1+t_1 + \cdots + t_m) u} u^{m-\alpha-1} \, dt_1\, \ldots\, dt_m \,du \\
    &= \frac{(-1)^m}{\Gamma(-\alpha)} \int_{0}^{\infty} e^{-(n+1)u} \left( \int_{0}^{1} e^{-tu} \,dt\right)^{m} u^{m-\alpha-1} \, du \\
    &= \frac{(-1)^m}{\Gamma (-\alpha)} \int_{0}^{\infty} e^{-(n+1)u} (1-e^{-u})^{m} u^{-\alpha-1} \, du.
\end{align*}

Substituting $t=e^{-u}$, we have
\begin{equation*}
    \Delta^m (n+1)^{\alpha} = \frac{(-1)^m}{\Gamma(-\alpha)} \int_{0}^{1} t^n (1-t)^m \left( -\log t\right)^{-\alpha -1} \, dt.
\end{equation*}

Define \begin{equation*}
    d\mu_m (z) = \frac{(-1)^m}{\Gamma(-\alpha)} (1-\abs{z}^2)^{m} (-\log |z|^2)^{-\alpha-1} \, dA(z)
\end{equation*}

We now show that
\begin{equation*}
    \int_{\D} z^k \overline{z}^l \, d\mu_m (z) = \Delta^m (k+1)^{\alpha} \delta_{k,l}.
\end{equation*}

Since $\mu_m$ is a radial measure, the conclusion is evident when $k \ne l$. We may assume $k=l$ where $k\in \Z_{\ge 0}$. Then 
\begin{align*}
    \int_{\D} \abs{z}^{2k} \, d\mu_m (z) &= \frac{(-1)^m}{\Gamma(-\alpha)} \int_{0}^{1} r^{2k} (1-r^2)^{m} (-\log r^2 )^{-\alpha -1 } (2r) dr \\
    &= \frac{(-1)^m}{\Gamma(-\alpha)} \int_{0}^{1}t^{k}(1-t)^m (-\log t)^{-\alpha-1}\, dt\\
    &= \Delta^{m}(k+1)^{\alpha}.
\end{align*}
This completes the proof of our claim.
\end{example}

 The question of finding a qualitative characterization of when a $m$-tuple $\overrightarrow{\mu} = \left( \mu_{0}, \mu_{1}, \ldots , \mu_{m-1} \right)\in (\mathcal D')^{m-1}\times M_+(\overline{\D})$ is allowable seems to be a bit challenging at the moment.  However to furnish further examples of ``allowable"  $m$-tuple $\overrightarrow{\mu}= (\mu_0,\mu_1,\ldots,\mu_{m-1})\in (\mathcal D')^{m-1}\times M_+(\overline{\D}),$ we state the following theorem, which extends \cite[Theorem 6.4]{Rydhe19}. In view of \cite[Lemma 5.8]{Rydhe19},  Proposition \ref{prop:sigma-vs-mu}, Proposition \ref{prop:boundedness-of-mu} mentioned here, the proof follows verbatim from the arguments presented in the aforementioned reference. Accordingly, we omit the details and refer the interested reader to \cite[Theorem 6.4]{Rydhe19} for a proof. In the following, for a distribution $\mu\in \mathcal D',$ we use the symbol $D\mu$ to denote the distribution in $\mathcal D'$ determined by $\widehat{D(\mu)}(k)= |k| \hat{\mu}(k),$ for $k\in \mathbb Z.$ 

\begin{theorem}
Let $m \ge 4$, $\overrightarrow{\mu}= (\mu_0,\dots,\mu_{m-2}) \in (\mathcal D')^{m-1},$ $\mu_{m-1}\in M_+(\overline{\D})\setminus \{0\},$ and 
$\nu \in \mathcal D'$. Assume:

    \begin{enumerate}[label=(\roman*)]
        \item $\nu \ge 0$ and $\abs{\hat f(0)}^2 \lesssim D_{\nu,0}(f)$ for $f\in \mathbb C[z],$ 
        \item If $0 \le n \le \frac{m-4}{2}$, then $\abs{\hat{\mu_n}(k)} \lesssim (1+|k|)^{(m-4)/2},$ for every $k \in \mathbb Z,$       
        \item If $m$ is odd and $n=\frac{m-3}{2}$, then $\mu_n=\sum_{j=0}^{\frac{m-3}{2}} D^j \nu_j,$ where each $\nu_j$ is a finite measure, 
        \item If $\frac{m-2}{2} \le n \le m-3$, then $\mu_n=\sum_{j=0}^{m-3-n} D^j \nu_j,$
        where each $\nu_j$ is a finite measure. In particular, $\mu_{m-2}=0$.
    \end{enumerate}
Then, for sufficiently large $C>0$, the tuple $\overrightarrow{\mu}= (\mu_0+C\nu,\mu_1,\dots,\mu_{m-2},C\mu_{m-1})$ is allowable.
\end{theorem}

 We note down another special class of normalized allowable $(m+1)$-tuple. Suppose each $\mu_j$ is a positive distribution in $\mathcal D',$ that is, $\mu_j \in M_+(\mathbb T)$ for $j=1,2,\ldots,m-1$ and $\mu_m\in M_+(\overline{\D}).$ Then in view of \cite[Proposition 5.7, Lemma 5.8]{Rydhe19} and the Remark \ref{zf vs f}, it follows that the $(m+1)$-tuple  $\overrightarrow{\mu} = \left( \sigma, \mu_{1}, \ldots ,\mu_{m} \right)$ is always a normalized allowable $(m+1)$ tuple. Moreover in this case $\mathscr D_{\overrightarrow{\mu}}$ can be described in terms of weighted Dirichlet type spaces in the following manner. 
 
 Consider $\mathscr H_{\overrightarrow{\mu}},$ the linear subspace of $\operatorname{Hol} (\D),$ defined by 
\begin{align*}
   \mathscr H_{\overrightarrow{\mu}}:= \bigcap\limits_{j=1}^m \mathscr D_{\mu_j,j} = \left \{ f \in \operatorname{Hol} (\D) \, : \, D_{\mu_j, j} (f) < \infty, \,\,\mbox{for every} \,\,j=1,2,\ldots,m\right \}.
 \end{align*}
Since $\mu_j\in M_+(\mathbb T)$ for each $j=1,2,\ldots,m-1,$ and $\mu_m\in M_+(\overline{\D}),$ we have $\mathscr D_{\mu_j,j}\subseteq H^2$ for every $j=1,2,\ldots,m.$ Consequently, $\|\cdot \|_{_{\overrightarrow{\mu}}}$ given by 
\begin{align*}
    \|f\|^2_{_{\overrightarrow{\mu}}}:= \|f\|_{H^2}^2+ \sum\limits_{j=1}^{m} D_{\mu_j,j}(f), \quad f\in  \mathscr H_{\overrightarrow{\mu}},
\end{align*}
defines a norm on $\mathscr H_{\overrightarrow{\mu}}.$ It follows that $\mathscr H_{\overrightarrow{\mu}}$ is a Hilbert space with respect to the norm $\|\cdot \|_{_{\overrightarrow{\mu}}}.$ In view of Proposition \ref{prop:dilations}, we also have  $\|f_r-f\|_{_{\overrightarrow{\mu}}} \to 0$ as $r\to 1^{-}$ for each $f\in\mathscr H_{\overrightarrow{\mu}}.$ For $0<r <1$ and $f\in \mathscr H_{\overrightarrow{\mu}},$ the $r$-dilation $f_r$ is holomorphic in a neighborhood of $\overline{\D}.$ Let $s_n(f_r)$ denotes the degree $n$ Taylor polynomial of $f_r$ around zero. Since for each $j\in \mathbb Z_{\geqslant 0},$ the polynomial $s_n(f_r^{(j)})$ converges uniformly to $f_r^{(j)}$ in a neighborhood of $\overline{\D},$ it follows that $\|s_n(f_r)-f_r\|_{_{\overrightarrow{\mu}}} \to 0$ as $n\to \infty.$ Thus we obtain that $\mathbb C[z]$ is dense $\mathscr H_{\overrightarrow{\mu}}$ and consequently we find that 
\begin{align*}
   \mathscr D_{\overrightarrow{\mu}}=\mathscr H_{\overrightarrow{\mu}},\qquad  \overrightarrow{\mu} = \left( \sigma, \mu_{1}, \ldots ,\mu_{m} \right) \in \left(M_+(\mathbb T)\right) ^m \times M_+(\overline{\D}).
\end{align*}
A bounded operator $T\in \mathcal B(\mathcal H)$ is called analytic if $\bigcap\limits_{n=1}^{\infty} T^n(\mathcal H)= \{0\}.$
As $\mathscr H_{\overrightarrow{\mu}}\subset \operatorname{Hol} (\D),$ the operator $M_z$ on $\mathscr H_{\overrightarrow{\mu}}$ turns out to be analytic. In view of Remark \ref{positivity of measure}, it follows that $M_z$ satisfies the operator inequality
\begin{align*}
    \beta_r(M_z)\geqslant \sum\limits_{j=1}^{\infty} {L^{*}}^j \beta_{r+1}(M_z) {L}^j,\qquad r=1,2,\ldots,m-1.
\end{align*}
The above inequality also characterize the operator $M_z$ on $\mathscr H_{\overrightarrow{\mu}}$ in a unique fashion. As a corollary to Theorem \ref{Model Theorem}, We note down the following model theorem for a class of analytic Dirichlet type operator of finite order which generalize \cite[Theorem 1.2]{GGR22} in a natural way. For a left invertible operator $T\in \mathcal B(\mathcal H),$ we use $L_T$ to denote the Moore-Penrose inverse of $T$ given by $L_T=(T^*T)^{-1}T^*.$
\begin{theorem}\label{analytic model}
    Let $m\in \mathbb Z_{\geqslant 1}$ and $T\in \mathcal B(\mathcal H)$ be an analytic, Dirichlet type operator of order $m$ satisfying 
    \begin{itemize}
        \item[\rm (i)] dim $(\ker T^*)=1$, 
        \item[\rm (ii)] $\beta_r(T)\geqslant \sum\limits_{j=1}^{\infty} {L_T^{*}}^j \beta_{r+1}(T) {L_T}^j,$ for each $r=1,2,\ldots,m-1.$
    \end{itemize}
    Then there exists a unique allowable $(m+1)$-tuple  $\overrightarrow{\mu} = \left( \sigma, \mu_{1}, \ldots ,\mu_{m} \right)$ in $\left(M_+(\mathbb T)\right) ^m \times M_+(\overline{\D})$ such that $T$ is unitarily equivalent to the operator $M_z$ on $\mathscr H_{\overrightarrow{\mu}}.$
\end{theorem}
\begin{proof}
    Since $T$ is a Dirichlet type operator of order $m,$ we have that $T$ is $(m+1)$-concave, that is, $\beta_{m+1}(T)\leqslant 0.$ By our assumption on $T,$ it follows from \cite[Theorem 1.3]{GGR22} that $T$ admits a unit cyclic vector, say $e,$ in $\ker T^*.$ Therefore by Theorem \ref{Model Theorem}, there exists an allowable $(m+1)$-tuple  $\overrightarrow{\mu} = \left( \mu_0, \mu_{1}, \ldots ,\mu_{m} \right)$ in $\left (\mathcal D'\right) ^m \times M_+(\overline{\D})$ such that $T$ is unitarily equivalent to the operator $M_z$ on $\mathscr D_{\overrightarrow{\mu}}.$ Since $e\in \ker T^*,$ we have $\hat{\mu_0}(k)= \delta_{0,k},$ for every $k\in \mathbb Z.$
    
    This gives us $\mu_0=\sigma,$ the normalized Lebesgue measure on $\T.$ Also note that $\hat{L}(p(T)e)= L_T(p(T)e)$ for every $p\in\mathbb C[z].$ Therefore it follows from Remark \ref{positivity of measure}, that $\mu_j\in M_+(\mathbb T)$ for every $j=1,2,\ldots,m-1.$ Thus we obtain that $\mathscr D_{\overrightarrow{\mu}}= \mathscr H_{\overrightarrow{\mu}}.$
    
    For the uniqueness part assume that $M_z$ on $\mathscr H_{\overrightarrow{\mu}}$ is unitarily equivalent to $M_z$ on $\mathscr H_{\overrightarrow{\nu}}$ for some other allowable tuple  $\overrightarrow{\nu} = \left( \sigma, \nu_{1}, \ldots ,\nu_{m} \right)$ in  $\left(M_+(\mathbb T)\right) ^m \times M_+(\overline{\D}).$ Let $U: \mathscr H_{\overrightarrow{\mu}} \to \mathscr H_{\overrightarrow{\nu}}$ be the associated unitary such that $UM_z=M_zU.$ Note that $\ker M_z^*=\mathbb C$ in both the spaces $\mathscr H_{\overrightarrow{\mu}}$ and $\mathscr H_{\overrightarrow{\nu}},$ therefore we must have $U(1)= \xi$ for some $\xi\in \mathbb T.$ Therefore $V=\bar{\xi}U$ is also an unitary from $\mathscr H_{\overrightarrow{\mu}}$ onto $\mathscr H_{\overrightarrow{\nu}}$ such that $V1=1$ and $VM_z=M_zV.$ By the uniqueness part of Theorem \ref{Model Theorem}, we then obtain that $\overrightarrow{\mu}=\overrightarrow{\nu}.$
\end{proof}
The uniqueness of the $(m+1)$-tuple  $\overrightarrow{\mu} = \left( \sigma, \mu_{1}, \ldots ,\mu_{m} \right) \in \left(M_+(\mathbb T)\right) ^m \times M_+(\overline{\D}),$ in Theorem \ref{analytic model} suggests that the operator $M_z$ on $\mathscr H_{\overrightarrow{\mu}}$ serves as a canonical model for a class of cyclic, analytic, Dirichlet type operator of finite order on a Hilbert space. The uniqueness of this model can also be described in another way. For an $(m+1)$-tuple  $\overrightarrow{\mu} = \left( \sigma, \mu_{1}, \ldots ,\mu_{m} \right) \in \left(M_+(\mathbb T)\right) ^m \times M_+(\overline{\D}),$ the associated Hilbert space  $\mathscr H_{\overrightarrow{\mu}}\subseteq H^2$ and we have $\|f\|_{H^2} \leqslant \|f\|_{\overrightarrow{\mu}}.$ Therefore it follows that $\mathscr H_{\overrightarrow{\mu}}$ is a reproducing kernel Hilbert space on $\mathbb D$, that is, the map $ev_z:\mathscr H_{\overrightarrow{\mu}} \to \mathbb C$ given by $ev_z(f)=f(z)$ is continuous for each $z\in \D,$ see \cite{zbMATH06562440} for basic properties of reproducing kernel Hilbert spaces. Moreover, the kernel $K(z,w)$ associated to $\mathscr H_{\overrightarrow{\mu}}$ turns out to be normalized, that is, $K(\cdot,0)=1,$ or equivalently, 
\begin{align*}
\ip{f,1}_{\overrightarrow{\mu}}=f(0),\qquad f\in \mathscr H_{\overrightarrow{\mu}}.
\end{align*}
This shows that the model for the class of analytic Dirichlet type operator as described in Theorem \ref{analytic model} is a canonical model and coincides with the model described by Shimorin in \cite[Page~154]{zbMATH01572594} for an arbitrary left invertible analytic operator on a Hilbert spaces, see also \cite[Lemma 4.8]{zbMATH03933844} for the uniqueness of the normalized kernel among all equivalent reproducing kernels. 

\section{Concluding Remarks}

The model theorem of Aleman for cyclic, analytic, Dirichlet type operators of order 1 provides a Beurling-type description of every closed $M_z$ invariant subspace of $D(\mu)$ for $\mu\in M_+(\overline{\D}),$ see \cite[IV, Theorem 4.9]{1993multiplication}. It is therefore natural to ask whether an analogous Beurling-type characterization can be obtained for all closed $M_z$-invariant subspaces of $\mathscr D_{\overrightarrow{\mu}},$ where ${\overrightarrow{\mu}}$ is an allowable tuple. The model theorem for cyclic Dirichlet-type operators of finite order developed in this work provides a systematic framework and serves as a foundation for further developments in this direction.

In this work, we have focused on higher order weighted Dirichlet-type integrals $D_{w,k}(f)$ with weights $w$ belonging to a specific subclass of positive $k$-superharmonic functions on $\D$, namely those of the form $w=U_{\mu},$ where $\mu\in M_+(\overline{\D}).$ This restriction is motivated by our aim to model a particular class of operators, namely Dirichlet-type operators of finite order. Nevertheless, it is of independent interest to understand which aspects of the theory persist for more general positive $k$-superharmonic weights. 

It is worth noting that every positive $k$-superharmonic function $w$ admits a decomposition of the form $w(z) = S_{\mu,k}(z) + v(z),$  where $S_{\mu,k}$ is a ``pure" $k$-superharmonic component (in the sense that it carries no $k$-harmonic part), and $v$ is $k$-harmonic function on $\D$. Indeed, if we define $\mu(z)= (1-|z|^2)^k (-\triangle_z)^k w(z)$ (in the distributional sense), then it follows from Proposition \ref{Basic properties of Green's function} (f), that 
\begin{align*}
   (-\triangle_z)^k S_{\mu,k}(z)= \frac{d\mu(z)}{(1-|z|^2)^k}= (-\triangle)^k w(z). 
\end{align*}
Thus $w(z)-S_{\mu,k}(z)$ defines a $k$-harmonic function on $\D.$ This decomposition suggests a natural direction for further investigation, namely the study of weighted Dirichlet-type integrals corresponding to a $k$-harmonic weights. Furthermore, by the Almansi decomposition, any $k$-harmonic function $v$ can be expressed as $v(z) = \sum_{j=0}^{k-1} (1-|z|^2)^j h_j(z),$
where each $h_j$ is harmonic in $\mathbb{D}.$ This representation naturally leads to the consideration of weights of the form
$w(z) = (1-|z|^2)^j h(z),$ with $h$ harmonic, as fundamental building blocks in the study of general $k$-superharmonic weights.

In this direction E. Rydhe \cite[Sec. 5]{Rydhe19} studied the weighted Dirichlet integral $D_{w,k}(f),$ with $w(z)=(1-|z|^2)^{k-1}P_{\nu}(z),$ where $\nu\in \mathcal D'.$ These weights play a crucial role in the description of model spaces associated with cyclic $m$-isometric operators. A systematic study of weighted Dirichlet-type integrals $D_{w,k}(f)$ corresponding to weights of the form $w(z) = (1-|z|^2)^j h(z),$ , with $h$ harmonic and $j<k-1,$ remains an interesting direction for future research and is expected to shed further light on the structure of higher-order Dirichlet-type operators and their invariant subspace structure. 

\section*{Declarations}
\subsection*{Funding} Ashish Kujur is supported through the Senior Research Fellowship of the Council of Scientific and Industrial Research (CSIR), India, (Ref. No. 09/0997(18166)/2024-EMR-I).

\subsection*{Competing interests} The authors declare that they have no competing interests.
\subsection*{Availability of data} No new data were created or analyzed in this study. Data sharing is not applicable.

\bibliographystyle{amsplain}
\nocite{*}
\bibliography{main}

\end{document}